\documentclass[10pt]{article}

\usepackage{color}
\usepackage{dsfont}
\usepackage{enumerate}
\usepackage{graphicx,float}
\usepackage{caption}
\usepackage{subcaption}
\usepackage{setspace}
\usepackage{hyperref}
\usepackage{color}
\usepackage{diagbox}
\usepackage[english]{babel}
\usepackage{amsmath, amsfonts, amssymb, amsthm, amscd,graphicx}

\newcommand{\comment}[1]{}

\newtheorem{theorem}{Theorem}

\newtheorem{assumption}{Assumption}

\newtheorem{corollary}{Corollary}

\newtheorem{definition}{Definition}

\newtheorem{lemma}{Lemma}

\newtheorem{proposition}{Proposition}
\newtheorem{remark}{Remark}

\newcommand{\bee}{\begin{equation}}
\newcommand{\eee}{\end{equation}}
\newcommand{\bea}{\begin{eqnarray}}
\newcommand{\eea}{\end{eqnarray}}
\newcommand{\bean}{\begin{eqnarray*}}
\newcommand{\eean}{\end{eqnarray*}}

\begin{document}
\nocite{*}

\title{ Scaling limit for stochastic control problems in population dynamics}
\author{ Paul Jusselin\footnote{paul.jusselin@polytechnique.edu}~~and Thibaut Mastrolia\footnote{thibaut.mastrolia@polytechnique.edu},\\
		\'Ecole Polytechnique, CMAP
		 }
\date{\today }
\maketitle

\begin{abstract}
\noindent Going from a scaling approach for birth/death processes, we investigate the scaling limit of solutions to non-Markovian stochastic control problems by studying the convergence of solutions to BSDEs driven a sequence of converging martingales. In particular we manage to describe how the values and optimal controls of control problems converge when the models converge towards a continuous population model.
\end{abstract}

\vspace{1em}


\section{Introduction}
The sustainability of natural resources has become a major subject of interest in the last decades for public institutions. For instance, in 1983 the European Union has launched its common fisheries policy to manage European fish stocks. In August 2010, a report of European commission named \textit{Water Scarcity and Drought in the European Union,} has emphasized that "an adequate supply of good-quality water is a pre-requisite for economic and social progress, so we need to do two things: we must learn to save water, and also to manage our available resources more efficiently".  A large part of academic literature has dealt with such issues. For example, Reed in \cite{reed1979optimal}, Clarke and Kirkwood in \cite{clark1986uncertain}, Regnier and De Lara in \cite{regnier2015robust} or Tromer and Doyen in \cite{tromeur2019optimal} have studied the exploitation of a natural resource under uncertainty on its evolution in a multi-period model. May, Beddington, Horwood and Sherpherd in \cite{may1978exploiting} have considered the problem by assuming that the intrinsic population growth rate is given by the difference between recruitment and mortality for general recruitment functions. These models have been extended to stochastic differential equations driven by a Brownian motion (see for instance the work of Saphores \cite{saphores2003harvesting}). Evans, Hening and Shreiber in \cite{evans2015protected} or more recently Kharroubi, Lim and Ly Vath in \cite{lim2018optimal} have modelled the dynamic of the natural resource as the solution of the logistic stochastic differential equation to solve a control problem under interaction between species and delayed renewal of the resource. We also refer to the book \cite{de2008sustainable} for stochastic and deterministic models and resource management problems. All the models mentioned above use a Brownian motion to describe the uncertainty of the system evolution. We refer to this class of model as \textit{continuous models}. On the other side of the literature, Getz in \cite{getz1975optimal} has studied control problems related to a birth/death process. This work has been extended more recently by Claisse in \cite{claisse2018optimal} to branching processes. We refer to those models as \textit{discrete models}.\\

It is well known that some continuous population models can be seen as scaling limits of discrete models, see for example the work of Bansaye and M\'el\'eard in \cite{bansaye2015stochastic}. Hence continuous models can be considered as good  approximations of the macroscopic evolution of a population size. Therefore it is relevant to consider continuous models for resources management purposes. Moreover those models are attractive from a tractability viewpoint compared to discrete models. Indeed solving control problems in Brownian driven model essentially boils down to solve a partial differential equation. Whereas for discrete models it leads to integral-partial differential equation, which are often more complex to solve. Yet the remaining question is the relevancy of designing a management policy based on a continuous modeling while the controlled population (or resource) is naturally discrete.\\

To try to give an answer to this question we are going to consider a sequence of discrete population models that converges towards a continuous population model. For each of those models we consider a control problem. Each of them are the natural adaptations of the same control problem to the different models. Therefore we expect the solutions of the discrete control problems to converge towards the solution of the continuous limit problem. From $\Gamma$-convergence results adapted to stochastic control problems as in for instance the articles of Buttazzo and Del Maso \cite{buttazzo1982gamma} and Belloni, Buttazzo and Freddi \cite{belloni1993completion}, we expect to have convergence of value functions (see also for instance \cite[Theorem 10.22]{dal2012introduction}) and a kind of weak convergence of optimal controls (see for instance \cite[Proposition 2.8]{belloni1993completion}).  This is emphasized in a toy model (see Section \ref{sec:definition_model}) where besides convergence of the value functions we also get convergence in law of the state process under the optimal control. In this paper we prove the convergence of the controls as sequence of processes. This is stronger than $\Gamma-$convergence. Since we aim at dealing with non-Markovian stochastic control problems our problematic is to prove the convergence of solutions to a sequence of Backward Stochastic Differential Equations (BSDE for short) driven by a sequence of converging martingales. \\

We know from the seminal paper of Donsker \cite{donsker1951invariance} that a scaling in time procedure leads to the weak convergence of a random walk to a Brownian motion. {Following the idea of Donsker, Pardoux in \cite{pardoux1999homogeneization,pardoux1999bsde} has studied the weak convergence of Markovian BSDE driven by Brownian motions. Extending this result to the theory of non-Markovian BSDE with fixed time horizon $T>0$, Briand, Delyon and M\'emin in \cite{briand2001donsker} have provided a time discretization of the Brownian motion to get the convergence of a time discretized BSDE. More precisely, they consider a sequence of random walks $(W^n)_{n\geq 0}$ converging towards a Brownian motion $W$. Then they prove the convergence of the solutions of a sequence of BSDEs driven the $(W^n)_{n\geq 0}$ towards the solution of a BSDE driven by $W$. The main idea is to prove the convergence of the terms involved in the martingale representation with respect to $W^n$ when $n$ goes to infinity. For this they use the convergence, in the sense of Coquet, M\'emin and Slominski \cite{coquet2001weak}, of the filtrations associated to each of the $(W^n)_{n\geq 0}$ towards the natural filtration associated to $W$. Those results have been extended in \cite{briand2002robustness} to a more general situation, without assuming that $W^n$  has a predictable representation property, but assuming that the brackets of the martingales $( W^n )_{n\geq 0}$ are uniformly bounded. In a more general case, preliminary results on the convergence of BSDEs driven by general martingales have been obtained in the PhD thesis of Alexandros Saplaouras, in $L^2$ for the Skorokhod distance, see \cite[Chapter IV]{alex}. \\

In the present paper we aim at extending both the results of \cite{pardoux1999homogeneization,briand2001donsker} to a wider class of martingale convergence and the results of \cite{alex} to stronger convergence. Starting from a scaling result in \cite{bansaye2015stochastic} showing that a sequence of scaled birth/death processes $(X^K)_{K\geq 0}$ with scaling parameter $K>0$ converges in law to the solution $X$ of a stochastic Feller diffusion, we begin to extend it to more general birth and death intensities. We then consider a sequence of BSDEs of the form
$$
\mathbf{(B)}_K:~ Y^K_t = \xi^K + \int_t^T g^K(X^K_s, Y_s^K, Z^K_s)\cdot \phi^K_s \mathrm{d}A^K_s - \int_t^T Z^K_s \cdot \mathrm{d}M^K_s,
$$
where $\xi^K$ is some general terminal random condition, $g^K$ the generator of the BSDE, $M^K$ a two dimensional martingale associated to the population model $X^K$ and $\phi^K\mathrm{d} A^K$ denotes the measure associated to the angle bracket of $M^K$. We also consider the continuous counterpart of $(\textbf{B})_K$,
$$
(\mathbf{B}):~ Y_t = \xi + \int_t^T g(X_s, Y_s, Z_s)\mathrm{d}A_s - \int_t^T Z_s \mathrm{d} M^X_s,
$$
where $\xi$ is some terminal condition, $g$ is the generator of the BSDE, $M^X$ is a one dimensional martingale related to the diffusion term of $X$ and $A$ is its angle bracket. The existence and uniqueness of solutions to such BSDEs driven by general martingales have been studied, for instance, by El Karoui and Huang in \cite{el1997general}, Confortola and Fuhrman in \cite{confortola2013backward} or more recently by Papapantoleon, Possama\"i and Saplaouras in \cite{papapantoleon2018existence} in a general framework. Inspired by \cite{briand2002robustness}, we prove that when $(\xi^K)_{K\geq 0}$ converges toward $\xi$ and $(g^K)_{K\geq 0}$ towards $g$, the solution of $(\mathbf{B})_{K}$ converges to the solution of $(\mathbf{B})$. Moreover, since we are motivated by applying our results to stochastic control problems, we prove a stronger convergence, compared to \cite{alex}, for the solutions to the particular BSDEs considered. To obtain such stronger convergence, we however need to set stronger assumptions on the integrability of terminal conditions $(\xi^K)_{K\geq 0}$. The methods used are related to the so-called martingale problem as stated by Jacod and Shiryaev in \cite{jacod2013limit} and to the double-Picard iterations craftly used in \cite{briand2002robustness}. Our results enable us to investigate the scaling limit of solutions to non-Markovian stochastic control problems applied to population monitoring. \\

\paragraph{Main Contributions.} Our approach extends the results in \cite{pardoux1999bsde,pardoux1999homogeneization,briand2001donsker,briand2002robustness} to counting processes, beyond the convergence of a time discretization of the BSDE. This paper is also more focused on applications to stochastic control problems, dealing with examples in population monitoring. We prove that a sequence of solutions to BSDEs driven by birth/death processes converges to a BSDE driven by a one dimensional Brownian motion in tractable spaces to investigate the scaling limits of solutions to stochastic control problems. This also extends the preliminary results on the robustness of solutions to BSDEs in \cite[Chapter IV]{alex} to the convergence of solutions to our particular BSDE in different spaces. Moreover, the results obtained on the convergence of optimal controls in stochastic control problems go beyond $\Gamma-$convergence since we get a strong form of convergence for the optimal controls.\\

The structure of the paper is the following.  In Section \ref{sec:definition_model} we study the convergence of a rescaled birth/death process to the solution of a stochastic Feller type SDE by extending \cite{bansaye2015stochastic} to more general dynamics (see Theorem \ref{th:rescaling}). We also  provide fundamental properties of our state processes such as exponential moments (see Proposition \ref{prop:expo_moments} and Corollary \ref{cor:expo:feller}). Section \ref{sec:toy_model} introduces a toy model motivating our study and illustrated with numerical simulations. In Section \ref{sec:convergence}, Theorem \ref{th:main_theorem} gives convergence of the solutions to $(\mathbf{B})_K$ to the solution of $(\mathbf{B})$. In Section \ref{sec:application_control} from our BSDE approach we deduce convergence of the values and optimal controls to a sequence of control problems. Section \ref{sec:proof} gives the main proofs of our results. Minor proofs are given in the appendix.\\

The spaces considered are defined in Appendix \ref{appendix:space}. 

\section{From a discrete to a continuous population model}
\label{sec:definition_model}

In this section we define a sequence of discrete population models. We show that this sequence converges in law towards a continuous Feller population model by extending \cite[Theorem III-3.2]{bansaye2015stochastic} to more general population dynamic models.

\subsection{Definition of the discrete population models}
\label{subsec:definition_discrete_models}

We consider positive continuous functions $f^b$, $f^d$ and $\sigma$ defined from $\mathbb R$ into $\mathbb R^+$ that satisfy the following standing assumption.

\begin{assumption}
\label{assumption:model}~
\begin{itemize}
\item[(i)] The functions $f^b$, $f^d$ and $\sigma$ are null on $\mathbb{R}^-$ and there exists non negative constants $\nu$, $\mu$, $\eta$ and $\underline{\eta}$  such that  for any $x\in \mathbb{R}^+$
$$
f^{b}(x)\leq \nu x,~ f^d(x) \leq \mu x ,~ \underline{\eta}x \leq \sigma^2(x) \leq \eta( 1+x),
$$

\item[(ii)]$f:=f^b-f^d$ and $\sigma^2$ are Lipschitz continuous.
\end{itemize}
\end{assumption}

We note $\Omega_d$ the set of piecewise continuous increasing positive functions with jumps equal to $1$. We denote by $\mathbb{F}$ the natural filtration associated to the canonical process $(N^b, N^d)$ of $\Omega_d^2$.\\

For a fixed $K\geq 0$ and $n\geq 0$ we define a population model on the stochastic basis $(\Omega_d^2,\mathbb{F})$. The initial population is $K n$ and the processes $N^b$ and $N^d$ represent respectively the number of birth and death in the population. This means that when the process $N^b$ jumps there is a new individual in the population and when $N^d$ jumps there is one individual less in the population. Therefore at time $t$ the population size is $K n  + N_t^b-N_t^d$. As we are interested in the large population limit (which corresponds to $K$ large) we consider the rescaled population process 
$$
X^{K, n} = n + \frac{  N^b - N^d}{K}.
$$
We define the birth intensity in the model with parameter $K$ and initial population $n$ as
$$
\lambda^{K,n, b}_t = \lambda^{K, b}(X^{K,n}_{t-}) :=f^b(X^{K,n}_{t-})K + \frac{\sigma^2(X^{K,n}_{t-})}{2}K^2
$$
and the intensity of death
$$
\lambda^{K,n, d}_t = \lambda^{K, d}(X^{K,n}_{t-}) := f^d(X^{K,n}_{t-})K + \frac{\sigma^2(X^{K,n}_{t-})}{2}K^2.
$$

\begin{remark}
\label{remark:bansayeincluded}
 Note that $f^b(x)=\mu x$, $f^d(x)=\nu x$ and $\sigma^2(x)=\sigma^2 x$ satisfy Assumption \ref{assumption:model}. Consequently, the model studied in Theorem III-3.2 in \cite{bansaye2015stochastic} is included in the scope of this paper.
\end{remark}

Following Theorem 3.6 in \cite{jacod1975multivariate} there exists a unique probability measure $\mathbb{P}^{K, n}$ on $(\Omega_d^2, \mathbb{F})$ such that the processes
$$
M^{K,n, i}_t = N^i_t - \int_0^t \lambda^{ K,n,i}_s\mathrm{d}s , \text{ for }i\in\{b,d\}
$$
are local martingales. It means that under the probability $\mathbb{P}^{K,n}$ the process $N^b$ (resp. $N^d$) has intensity $\lambda^{K,n, b}$ (resp. $\lambda^{K, n,d}$). Note that if $m>n\geq 0$ then $\mathbb{P}^{K,n}$ is absolutely continuous with respect to $\mathbb{P}^{K,m}$ and we have:
\begin{equation}
\label{eq:change_proba_model}
\frac{\mathrm{d}\mathbb{P}^{K, n}}{\mathrm{d}\mathbb{P}^{K, m}} = L^{n,m}_T
\end{equation}
where
$$
\mathrm{d}L^{n,m}_t = L^{n,m}_{t-} \big( \sum_{i\in\{b,d\}} \frac{\lambda^{K,n,i}_t-\lambda^{K,m,i}_t}{\lambda^{K,m,i}_t} \mathbf{1}_{X^{K,m}_{t-} > 0} \mathrm{d}M^{K,m,i}_t \big)\text{ with }L^{n,m}_0 = 1.
$$
We justify this change of measure in Appendix \ref{appendix:change_measure}.\\

For the rest of this work we fix an initial population $x_0$ and do not write anymore the superscript $x_0$ to lighten the notations. We write $\mathbb E$ instead of $\mathbb E^{\mathbb P^K}$ when there is no ambiguity on the probability used. For any $K$ we consider the processes $M^K = (M^{K, b}, M^{K, d})$, $\lambda^K_t = \lambda^{K, b}_t-\lambda^{K, d}_t$ and for $i\in \{b, d\}$
$$
\overline \Lambda^{K,i}_t= \int_0^t \lambda^{K,i}_s K^{-2}\mathrm{d}s,~ \overline{N}^{K,i} = N^{i} K^{-2}\text{ and }\overline{M}^{K,i} = M^{K,i} K^{-1}.
$$
We note $\overline{M}^K = (\overline{M}^{K, b}, \overline{M}^{K, d})$. The rescaled population process is now noted
$$
X^K = x_0 + \frac{N^b - N^d}{K}.
$$

\subsection{Scaling limit of the sequence $(X^K)_{K\geq 0}$}
\label{subsec:scaling_limit}

Intuitively, and having in mind \cite[Theorem 3.2]{bansaye2015stochastic}, a continuous version of the processes $(X^K)_{K\geq 0}$, denoted by $X$, would be an Ito diffusion with drift equal to $f(X)$ and volatility given by $\sigma^2(X)$. We formalize this intuition in the following result  which extends \cite[Theorem 3.2]{bansaye2015stochastic}. The proof is given in Section \ref{appendix:rescalingthm}.
\begin{theorem}
\label{th:rescaling}
The sequence $\big(X^K, \overline{M}^{K}, \overline{N}^{K,b}, \overline{N}^{K,d}, \overline{\Lambda}^{K,b}, \overline{\Lambda}^{K,d} \big)_{K\geq 0}$ converges in law for the Skorokhod topology towards $ (X, M, A, A, A,A)$ such that
\begin{itemize}
\item[(i)] there exists a bi-dimensional Brownian motions $(B^b, B^d)$ satisfying
$$
M_t = \int_0^t \frac{\sigma(X_s)}{\sqrt{2}} \mathrm{d}(B^b_s, B^d_s),
$$
\item[(ii)] with $B = (B^b + B^d)/\sqrt{2}$, the process $X$ is the unique strong solution of
\[
(\mathbf{S}):~X_t = x_0 + \int_0^t f(X_s)\mathrm{d}s + \sigma(X_s)\mathrm{d}B_s,
\]
\item[(iii)] $A = \int_0^{\cdot}\sigma^2(X_s)\mathrm{d}s$.
\end{itemize} 

Moreover, there exists a probability space $(\Omega, \mathbb{F}, \mathbb{P})$ and a sequence $(N^{K,d}, N^{K,b})_{K\geq 0}$ such that for any $K\geq 0$, $(N^{K, d}, N^{K,b})$ has the law of $(N^d, N^b)$ under $\mathbb{P}^{K}$. Moreover on this space the sequence 
$$
(X^K, \overline{M}^{K}, \overline{N}^{K,b}, \overline{N}^{K,d}, \overline{\Lambda}^{K,b}, \overline{\Lambda}^{K,d} )_{K\geq 0}
$$
 converges in $\mathcal{S}^2_1 \times \mathcal{S}_2^2 \times \mathcal{S}_1^1 \times \mathcal{S}_1^1\times \mathcal{S}_1^1 \times \mathcal{S}_1^1$ to $(X, M, A, A, A, A)$ when $K$ goes to $+\infty$.
\end{theorem}

According to the last point of Theorem \ref{th:rescaling} from now on we work under the probability space $(\Omega, \mathbb{F}, \mathbb{P})$ and we consider that the processes  $
(X^K, \overline{M}^{K}, \overline{N}^{K,b}, \overline{N}^{K,d}, \overline{\Lambda}^{K,b}, \overline{\Lambda}^{K,d} )_{K\geq 0}$ and $(X, M, A, A, A, A)$ are defined on this space. For any $K$ we note $\mathbb{F}^K$ the natural filtration associated to $X^K$ and $\mathbb{F}^X$ the natural filtration associated to $X$.\\

Before going to the next section, we define some processes that we will extensively use in the rest of the paper. We note for $i\in\{b,d\}$
$$
M^i_t := \int_0^t \frac{\sigma(X_s)}{\sqrt{2}}\mathrm{d}B^i_s,
$$
so that $M=(M^d, M^b)$ and $M^X := M^b+M^d$ that is written
$$
M_t^X = \int_0^t \sigma(X_s)\mathrm{d}B_s.
$$
We also consider the processes
$$
A^K_t := \int_0^t \frac{\lambda^{K, b}_s + \lambda^{K, d}_s}{K^2} \mathrm{d}s,~p^{K, b}_t := \frac{\lambda^{K, b}_t}{\lambda^{K, b}_t + \lambda^{K, d}_t},\text{ and }p^{K, d}_t := \frac{\lambda^{K, d}_t}{\lambda^{K, b}_t + \lambda^{K, d}_t}.
$$
Note that under the probability $\mathbb{P}^{K}$ the random measure $m$ associated to the process $(N^b, N^d)$ interpreted as a compound jump process with values in $E = \{b, d\}$ admits as predictable compensator measure
$$
\pi^K(\mathrm{d}e, \mathrm{d}t) =  \big( \phi^{K,b}_t\delta_{b}(\mathrm{d}e) + \phi^{K,d}_t\delta_{d}(\mathrm{d}e) \big) \mathrm{d}A^K_t 
$$
with $\phi^K_t= (\phi^{K,b}_t,\phi^{K,d}_t)= (p^{K, b}_t, p^{K, d}_t)K^2$ and where $\delta_i$ denotes the Dirac measure at point $i\in  \{b, d\}$. This point of view is introduced in order to draw a parallel with the framework of \cite{confortola2013backward} to which we will refer extensively in Section \ref{subsec:convergence_BSDEs}.\\

\subsection{Uniform exponential moments}
\label{subsec:uniform_exponential_moments}

Finally we show that the sequence of processes $(X^K,\int_0^\cdot X_s^K)$ admits exponential moments uniformly in $K$ if $\sigma^2$ is linear. The proof of this result is postponed in Appendix \ref{proof:expo_moments}.
\begin{proposition}
\label{prop:expo_moments}
If there exists a positive constant $\sigma$ such that $\sigma^2(x) = \sigma^2 x$ there exists some positive constants $\beta_0$, $K_0$ and $T$ such that for any $s\leq T$ we have
$$
\underset{K\geq K_0}{\sup}~ \mathbb{E}[{\exp}(\beta_0 \int_0^s X^K_u\mathrm{d}u + \beta_0 X^K_s)]<\infty.
$$
\end{proposition}
Without loss of of generality we assume that $K_0=0$. From now we fix a positive constant $\beta$ strictly smaller than $\beta_0$. As a consequence of Proposition \ref{prop:expo_moments}, for any integer $q$ we have for any $s \leq T$
\begin{equation*}
\underset{K\geq 0}{\sup} ~ \mathbb{E}[{\exp}(\beta \int_0^s X^K_u\mathrm{d}u + \beta X^K_s)(1 + |X^K_s|^q)]<+\infty
\end{equation*}
and
\begin{equation*}
\underset{K\geq 0}{\sup} ~ \mathbb{E}[\int_0^ T {\exp}(\beta \int_0^s X^K_u\mathrm{d}u + \beta X^K_s)(1 + |X^K_s|^q)\mathrm{d}s]<+\infty.
\end{equation*}

We deduce from Fatou's Lemma together with Proposition \ref{prop:expo_moments} that $X$ inherits from the exponential moments of $X^K$ as stated in the following corollary.
\begin{corollary}\label{cor:expo:feller}
If there exists $\sigma$ positive such that $\sigma^2(x) = \sigma^2 x$ there exists some positive constants $\beta_0$ and $T$ such that for any $s\leq T$ we have
$$
\mathbb{E}[{\exp}(\beta_0 \int_0^s X_u\mathrm{d}u + \beta_0 X_s)]<\infty.
$$
\end{corollary}

In order to benefit from those exponential moments we now assume that $\sigma^2(x) =\sigma^2 x$ for some positive constant $\sigma$ fixed.

\section{Illustration of the study on a toy model}
\label{sec:toy_model}

In this section, we illustrate the $\Gamma-$convergence result applied to optimization problems in population dynamics. We consider specific parameters $f^d,~f^b$ and a sequence of control problems for which we are able to make explicit computations. Then, we show that the sequence of optimal controls converges in law to the optimal control of a continuous problem. In this section, we aim at providing the general main ideas of the paper rather than being perfectly accurate. Rigorous statements will be given in Section \ref{sec:application_control}.

\subsection{Discrete populations models}
\label{subsec:toy_model_discrete}

 We consider a discrete birth/death model as studied in \cite{bansaye2015stochastic} by choosing:
\begin{itemize}
\item[-] the initial population $x_0 \in \mathbb{N}$,
\item[-] the birth rate $f^b(x)=\nu x$ for some $\nu>0$, 
\item[-] the death rate $f^d(x)=\mu x$ for some $\mu>0$.
\end{itemize}
Recall from Remark \ref{remark:bansayeincluded} and from \cite[Theorem 3.2]{bansaye2015stochastic} that $(X_t^K)_{t \in [0,T]}$ converges in law for the Skorokhod topology towards the continuous diffusion process $(X_t)_{t \in [0, T ]}$ solution of the Feller stochastic differential equation
\begin{equation}
\label{feller:sto}
\mathrm{d}X_t = (\nu - \mu)X_t \mathrm{d}t + \sigma \sqrt{X_s}\mathrm{d}W_t,
\end{equation}
for $W$ a Brownian motion.\\

In this toy model, we assume that a resource manager regulates the population $X^K$ through an $\mathbb F^K$-predictable control $\alpha$. A control $\alpha$ is admissible if
\begin{itemize}
\item there exists a unique law $\mathbb{P}^{K, \alpha}$ under which the death intensity of the population is
\begin{equation*}
\lambda_t^{K,d,\alpha}:= KX_t^K(\mu+ K\frac{\sigma^2}{2}) + K X_t^K\alpha_t,
\end{equation*}
and the birth intensity is $\lambda^{K,b}$. When this probability exists, it is the law of the population under the control $\alpha$. 
\item $\lambda^{K,d, \alpha}$ is a non negative process $\mathbb{P}^{K, \alpha}$ almost surely. 
\end{itemize}
We denote by $\mathcal A^K$ the set of admissible controls.\\

The agent is assumed to be penalized if he fails at reaching a fixed level $\widetilde{x}>0$ of the resource at time $T$ determined by a regulator. We model this penalization by the square of the difference between the effective population size at time $T$ and the target $\tilde{x}$. So that the manager pays $\gamma(X_T^K-\widetilde{x})^2$ at time $T$ where $\gamma$ is a positive constant. The manager payoff is also assumed to be penalized by the instantaneous amount $\frac{|\alpha_t X_t^K|^2}{2}$ per unit of time when its effort is $\alpha$. The problem of the resource manager is thus to solve 
\begin{equation*}
(\textbf{TM})_K:~V_0^{K} = \underset{\alpha\in \mathcal{A}^K}{\sup} \; \mathbb{E}^{K, \alpha}[ -\gamma(X_T^K-\widetilde{x})^2 - \int_0^T \frac{ (\alpha_s X_s^K)^2}{2}\mathrm{d}s ]
\end{equation*}
where $\mathbb{E}^{K, \alpha}$ denotes the expectation taken under the probability $\mathbb{P}^{K, \alpha}$. We assume that $\sigma^2>2\gamma \tilde{x}$ and $\gamma<\mu$.\\

To solve this problem, as usual in stochastic control theory, we study the corresponding Hamilton-Jacobi-Bellman (HJB for short) equation and use a verification argument. The HJB equation associated to the control problem $(\textbf{TM})_K$ is
\[(\textbf{HJB})_K\begin{cases}
\partial_t U^K(t,x) +H^K\big(x,D_+^KU^K(t,x),D_-^KU^K(t,x)\big) = 0,\quad  (t,x)\in[0, T)\times (\mathbb{N}^*/K),\\
 U^K(T, x) = -\gamma(x-\widetilde{x})^2, \quad x\in (\mathbb N^*/K),
\end{cases}\]
with Hamiltonian $H^K$ given by
\[
H^K(x,p^+,p_-)=\sup_{\alpha} \Big\{ Kx (\nu+\frac{\sigma^2}2K) p_+ + Kx(\mu+\alpha+\frac{\sigma^2}2K)p_- - \frac{(\alpha x)^2 }{2}\Big\},
\]
and where
$$
D^K_+ U^K (t, x) = U^K (t, x+1/K) - U^K (t, x)\text{ and }~D^K_- U^K (t, x) = U^K (t, x-1/K) - U^K (t, x).
$$
The maximizer of the Hamiltonian is $\alpha^{K,*} = \frac{K p_-}{x}\mathbf 1_{x>0}$, hence
$$
H^K(x, p^+, p_-) = Kx (\nu+\frac{\sigma^2}2K) p_+ + Kx(\mu+\frac{\sigma^2}2K)p_- +  \frac{(Kp_-)^2}{2}\mathbf{1}_{x>0}.
$$
Note that we do not actually care about the value of the control when $x=0$ since if the population reaches $0$, it is stuck at this value. The partial differential equation (PDE for short) $(\textbf{HJB})_K$ is quadratic, so we search for a solution under the form
$$
U^K(t, x) = a_K(t) x^2 + b_K(t) x + c_K(t).
$$
Identifying the monomials, we get that $U^K$ is solution of $(\mathbf{HJB})_K$ if and only if $(a_K, b_K, c_K)$ is solution of the following systems of ODEs:
\begin{equation*}
\textbf{(ODE)}_K:~\left\{
\begin{array}{ll}
a_K'(t) + 2a_K(t) (\nu-\mu) + 2a_K^2(t) = 0,&~a_K(T) = -\gamma,\\
b_K'(t) -2 a_K(t)\big(\frac{a_K(t)}{K} - b_K(t) \big) + a_K(t) (\sigma^2+\frac{\mu+\nu}K) + b_K(t) (\nu-\mu) = 0,&~b_K(T) = 2\gamma  \widetilde{x},\\
c_K'(t) + \frac{1}{2  }\big( \frac{a_K(t)}{K} - b_K(t) \big)^2 = 0,&~c_K(T) = -\gamma \widetilde{x}^2.
\end{array} 
\right.
\end{equation*}
By Cauchy-Lipschitz theorem this system admits a unique solution. Thus the optimal effort of the agent is 
\[\alpha^{K,*}_t=\frac{1}{X_t^K}\Big(\frac{a_K(t)}{K}- 2X_t^Ka_K(t)-b_K(t)\Big)\mathbf{1}_{X^K_t>0} \]
and the corresponding death intensity is given by
$$
\lambda^{K,d,\alpha^{K,*}}_t=  KX_t^K(\mu+ K\frac{\sigma^2}{2}) + \big( -2X_t^Ka_K(t) +  \frac{a_K(t)X^K_t}K-b_K(t)\big)\mathbf{1}_{X_t^K>0}.
$$

Note that in view of $(\textbf{ODE})_K$ and since $a_K(T)$ is negative and $b_K(T)$ positive, there exists a $T$ small enough, independent of $K$, such that for any $K$ the control $\alpha^{K,*}$ is in $\mathcal{A}^K$. We refer to Appendix \ref{appendix:toy_model} for more details on this point. We assume that we are considering such short enough time horizon here.

\subsection{Continuous populations model}
\label{subsec:toy_model_continuous}

We now turn to the continuous version of the control problem. We assume that the manager controls the drift term in \eqref{feller:sto} through an $\mathbb F^X-$predictable process $\alpha$. We say that $\alpha$ is an admissible control when the following SDE admits a unique weak solution
$$
\mathrm{d}X_t = (\nu -\mu - \alpha_t)X_t \mathrm{d}t + \sigma \sqrt{X_t} \mathrm{d}W_t.
$$
When such solution exists we note $\mathbb{P}^{\alpha}$ its law that is the law of the population under the control $\alpha$. We denote by $\mathcal A$ the set of admissible controls. \\

The control problem in the continuous framework is written
\begin{equation*}
(\textbf{TM}):~V_0= \underset{\alpha\in \mathcal{A}}{\sup} \; \mathbb{E}^{ \alpha}[ -\gamma(X_T-\widetilde{x})^2 - \int_0^T \frac{ (\alpha_s X_s)^2}{2}\mathrm{d}s ].
\end{equation*}
The associated HJB equation is given by
\[(\textbf{HJB})\begin{cases}
\partial_t U(t,x) +H(x,DU(t,x),\Delta U(t,x)) = 0,\quad  (t,x)\in[0, T)\times \mathbb R^+,\\
 U(T, x) = -\gamma(x-\widetilde{x})^2, \quad x\in \mathbb R^+,
\end{cases}\]
where the Hamiltonian $H$ is
\begin{align*}
H(x,p,q)=\sup_{\alpha} \Big\{ (\nu-\mu-\alpha)xp-\frac{|\alpha x|^2}2+\frac12 x \sigma^2 q\Big\}=(\nu-\mu)xp+\frac12 x \sigma^2 q+\frac{p^2}2\mathbf{1}_{x>0}.
\end{align*}
The maximizer of the Hamiltonian is
$$
\alpha^*(x, p) = \frac{-p}{x}\mathbf{1}_{x>0}
$$ 
As previously, we are looking for a quadratic solution of the form
\[U(t,x)=a(t)x^2+b(t)x+c(t).\]
Identifying the monomials, we get that $U$ is solution of $(\textbf{HJB})$ if and only if $(a, b, c)$ is solution of the following system of ODEs.
\begin{equation*}
\textbf{(ODE)}:~\left\{
\begin{array}{ll}
a'(t) + 2a(t) (\nu-\mu) + 2a^2(t) = 0,&~a(T) = -\gamma,\\
b'(t) +2 a(t)b(t)  + a(t) \sigma^2 + b(t) (\nu-\mu) = 0,&~b(T) = 2\gamma \widetilde{x},\\
c'(t) + \frac{|b(t)|^2}2 = 0,&~c(T) = -\gamma \widetilde{x}^2.
\end{array} 
\right.
\end{equation*}

Hence, the optimal control is given by 
$$
\alpha^*_t=-\frac{DU(t,X_t)}{X_t}\mathbf{1}_{X_t>0}=-\frac{2a(t)X_t+b(t)}{X_t}\mathbf{1}_{X_t>0}.
$$
Note that in view of $(\textbf{ODE})$ and since $a(T)$ is negative and $b(T)$ positive, there exists a $T$ small enough such that the control $\alpha^{*}$ is in $\mathcal{A}$, see Appendix \ref{appendix:toy_model} for details. We assume that we are considering such time horizon here. We also note that $a_K=a$ and as consequence of Gr\"onwall Lemma $(b_K,c_K)_{K\geq 0}$ converges to $(b,c)$ when $K$ goes to $+\infty$. Consequently we get the convergence of the value of the control problems, $\lim\limits_{K\to+\infty}V_0^K=V_0$. Moreover a direct adaption of the proof of Proposition \ref{th:rescaling} gives the convergence in law of the optimally controlled population:
$$
\lim\limits_{K\to+\infty}\mathbb{P}^{K,\alpha^{K,*}} = \mathbb{P}^{\alpha^*}.
$$
Those convergences are illustrated in Figures \ref{fig:value} and \ref{fig:control} respectively.

\begin{figure}[tbph!]
\begin{center}
\includegraphics[width=8cm,height=6cm]{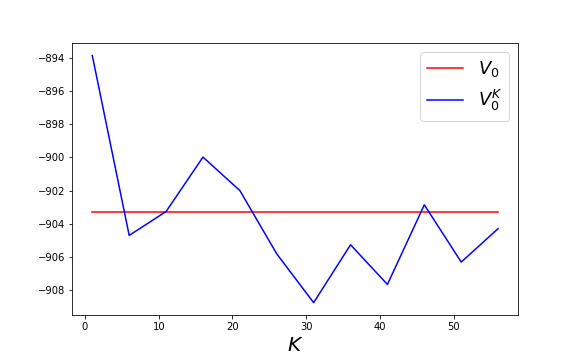}
\caption{Convergence of $(V_0^K)_{K\geq 0}$ towards $V_0$ with $\sigma^2 = 0.3$, $\mu = 0.1$, $\nu = 0.2$, $T = 0.1$, $x_0 = 50$, $\tilde{x} = 20$ and $\gamma = 1$.}
\label{fig:value}
\end{center}
\end{figure}

\begin{figure}[tbph!]
\begin{center}
\includegraphics[width=14cm,height=8cm]{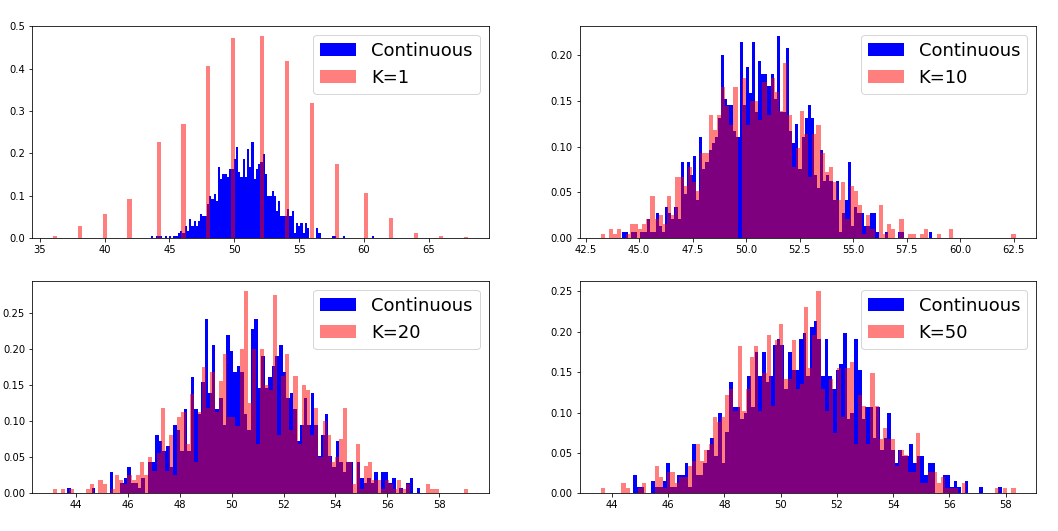}
\caption{Empirical distribution of the discrete optimal controls at time $t = 0.1$ for different values of $K$ (in red) compared to the distribution of the continuous optimal control. The parameters are the same than in Figure \ref{fig:value}}.
\label{fig:control}
\end{center}
\end{figure}

\section{Convergence of BSDEs}
\label{sec:convergence}

In this section we prove the main results of this paper. We recall a convergence result of a sequence of martingale representations given in \cite{papapantoleon201stability}. Then we extend it to the convergence of a sequence of BSDEs driven by the sequence of martingales $(M^K)_{K\geq 0}$.

\subsection{Convergence of martingale representations}
\label{subsec:convergence_representation}

From Theorem 2 in \cite{davis1976representation} we know any $\mathbb{F}^K-$martingale has the representation property with respect to $M^K$ (in the sense of Definition III-4.22 in \cite{jacod2013limit}). Moreover we prove in Appendix \ref{appendix:martingale_representation.} that any $\mathbb{F}^X-$ martingale has the representation property relative to $M^X$.\\

For any $K\geq 0$ we consider $\xi^K \in L^2$ an $\mathcal{F}_T^{K}$-measurable real random variable and $\xi\in L^2$ an $\mathcal{F}_T^{X}$-measurable real random variable. We define the closed martingale $Q^K$ by $Q^K_t = \mathbb{E}[\xi^K|\mathcal{F}^K_t]$, $\mathbb P-$a.s.. Since $Q^K$ is an $\mathbb F^K$-martingale and $\xi^K\in  L^2$, we know that there exists a unique process $Z^K\in L^2(M^K)$ such that 
$$
Q^K_t = \mathbb{E}[\xi^K|\mathcal{F}^K_t] = Q^K_0 + \int_0^t Z^K_s\cdot \mathrm{d}M^K_s.
$$
Similarly considering the $\mathbb F^X$-martingale $Q$ defined by $Q_t = \mathbb{E}[\xi|\mathcal{F}_t]$, $\mathbb P-a.s.$ since $\xi \in \mathbb{L}^2$ we have existence and uniqueness of $Z\in L^2(M^X)$ such that
$$
Q_t = \mathbb{E}[\xi|\mathcal{F}^X_t] = Q_0 + \int_0^tZ_s \mathrm{d}M^X_s.
$$
We have the following result as a mere extension of \cite[Theorem 3.3]{papapantoleon201stability}. 
\begin{proposition}[\textbf{martingale representations convergence}]
\label{prop:convergence_representation}
If the sequence $(\xi^K)_{K\geq 0}$ and $\xi$ are in $L^{2+\varepsilon}$ and $(\xi^K)_{K\geq 0}$ converges towards $\xi$ in $L^{2+\varepsilon}$ for $\varepsilon>0$ then
$$
\Big( Q^K, \langle Q^K, Q^K \rangle, \langle Q^K, \overline{M}^K \rangle \Big) \rightarrow \Big( Q, \langle Q , Q \rangle, \langle Q , M\rangle  \Big) \text{ as }K\rightarrow +\infty
$$
in $\mathcal{S}^{2+\varepsilon'}_1 \times \mathcal{S}^{1+\varepsilon'/2}_1\times \mathcal{S}^{2+\varepsilon'}_2$ for any $\varepsilon'\in [0,\varepsilon)$.
\end{proposition}

Compared to Theorem 5 in \cite{briand2002robustness} or Theorem 3.3 \cite{papapantoleon201stability} we have assumed that the convergence of $(\xi^K)_{K\geq 0}$ takes place in $L^{2+\varepsilon}$ instead in $L^2$. This is in order to extend the convergence of $( \langle Q^K, \overline{M}^K \rangle)_{K\geq 0}$ beyond $\mathcal{S}^1_2$. Indeed if we only assume that $\langle Q^K \rangle \in \mathcal{S}^1_1$, $\langle Q^K, \overline{M}^K \rangle$ is not squared integrable \textit{a priori}. In \cite{briand2002robustness} the authors do not face this issue since they assume that the brackets of the martingales they consider are bounded, see Hypothesis $(H1)$. In our framework the sequence $(\langle \overline{M}^K \rangle_T)_{K\geq 0}$ is not bounded in general. However, if we instead consider a sequence of models with a bounded population then $(\langle \overline{M}^K \rangle_T)_{K\geq 0}$ would be bounded and we could get the same result assuming only the convergence of $(\xi^K)_{K\geq 0}$ in $L^2$ only.\\

\subsection{Convergence of BSDEs}
\label{subsec:convergence_BSDEs}

We now extend the previous result to convergence of a sequence of BSDEs driven by $M^K$.\\

For any $K\geq 0$ we consider an $\mathcal{F}^K_T$ random variable $\xi^K$ and two continuous functions $g^{K}_b$ and $g^{K}_d$ from $\mathbb{R}^3 $ into $\mathbb{R}$. We write for $(x, y, z)\in \mathbb{R}\times \mathbb{R} \times \mathbb{R}^2$
$$
g^K(x, y, z) = \big( g_b^K(x, y, z^b), g_d^K(x, y, z^d) \big).
$$
Note that in the above equation, we implicitly use the decomposition $z=(z^b, z^d)$. We will always assume such convention when we are dealing with a pair of elements such that one element of the pair is related to the birth in the population and the other is related to the death.\\

We introduce the BSDE with generator $g^K$ and terminal value $\xi^K$ by setting
$$
(\mathbf{B})_K:~ Y^K_t = \xi^K + \int_t^T g^K(X^K_s, Y^K_s, Z^{K}_s) \cdot \phi^K_s \mathrm{d}A^K_s - \int_t^T Z^K_s \cdot \mathrm{d} M^K_s.
$$

\begin{definition}
\label{def:BSDE_K}
A solution to BSDE $(\mathbf{B})_{K}$ is a pair of processes $(Y,Z)\in \mathbb{S}^K$ such that the relation $(\mathbf{B})_{K}$ holds $\mathbb P-a.s.$
\end{definition}

As a consequence of Theorem 3.4 in \cite{confortola2013backward} we have the following result.

\begin{lemma}
\label{lemma:bsde_a}
Assume that
\begin{itemize}
\item[(i)] $\xi^K \in \mathbb{T}^K$,
\item[(ii)]there exists a positive constant $L$ such that $\beta > L^2+2L $ and for any $x, y, y', z, z'$ and $K \geq 0$ we have for $j\in\{b, d\}$
\begin{equation*}
K^2|g_j^K(x, y, z/K) - g_j^K(x, y', z'/K)|\leq L\big( |y-y'| + |z-z'|\big),
\end{equation*}
\item[(iii)] $g_j(X^K_t, 0,0)\in \mathbb{H}_1^K$ for $j\in \{b, d\}$,
\end{itemize}
then the BSDE $(\mathbf{B})_K$ has a unique solution $(Y^K, Z^K)\in \mathbb{S}^K$.
\end{lemma}

We also introduce a class of BSDE driven by the martingale $M^X$. For an $\mathcal{F}^X_T$ real valued random variable $\xi$ and a continuous function $g$ from $\mathbb{R}^3$ into $ \mathbb{R}$ we consider the BSDE
$$
(\mathbf{B}):~ Y_t = \xi + \int_t^T g(X_s, Y_s, Z_s)\mathrm{d}A_s - \int_t^T Z_s \mathrm{d} M^X_s.
$$

\begin{definition}
\label{def:BSDE}
A solution to BSDE $(\mathbf{B})$ is a pair of processes $(Y,Z)\in \mathbb{S}$ such that the relation $(\mathbf{B})$ holds $\mathbb P-a.s.$
\end{definition}

We get the following result on existence and uniqueness of solution to $(\mathbf{B})$ which is a consequence of Theorem 6.1 in \cite{el1997general} or Theorem 2.1 in \cite{carbone2008backward}.

\begin{lemma}
\label{lemma:bsde_b}
Assume that
\begin{itemize}
\item[(i)] $\xi\in \mathbb{T}$,
\item[(ii)] there exists a positive constant $L$ such that $\beta > L^2+2L $ and for any $x, y, y', z, z'$ we have:
$$
|g(x, y, z) - g(x, y', z')|\leq L\big( |y-y'| + |z-z'|\big),
$$
\item[(iii)] $g(X_t, 0, 0)\in \mathbb{H}$,
\end{itemize}
then the BSDE $(\mathbf{B})$ has a unique solution $(Y, Z)\in \mathbb{S}$.
\end{lemma}

We are interested in the convergence of the solutions to $(\mathbf{B})_K$ when $(\xi^K)_{K\geq K}$ and $(g^K)_{K\geq 0}$ converge. Therefore we make the following converging assumptions on the drivers of the BSDEs $(\mathbf{B})_K$.

\begin{assumption}
\label{assumption:bsde_b}~
 \begin{itemize}
\item[(i)] The sequence $(\xi^K)_{K\geq 0}$ converges towards $\xi\in \mathbb{T}$ in $L^{2+\varepsilon}$ for $\varepsilon>0$,
\item[(ii)] there exists a positive constant $C$ such that for any $x, x', y, z$, $K\geq 0$ and $j\in \{b, d\}$
$$
K^2|g_j^K(x, y, z) - g_j^K(x', y, z)|\leq C|x-x'|,
$$
\item[(iii)] there exists a pair of continuous functions $(g_b, g_d)$ from $\mathbb{R}^3$ and a positive sequence $(\upsilon_K)_{K\geq 0}$ converging towards $0$ such that for any $K \geq 0$, $x, y$ and $z$ we have for $j\in\{b, d\}$:
\begin{equation*}
|K^2g_j^K(x, y, z/K) - g_j(x, y, z) | \leq \upsilon_K (1 + x^2 + y^2 + \|z\|^2).
\end{equation*} 
\end{itemize}
\end{assumption}

\begin{remark}
\label{remark:bsde}
Under Assumption \ref{assumption:bsde_b} (iii) if for any $K$ the pair $(g_b^K, g^K_d)$ satisfies assumptions (ii) and (iii) in Lemma \ref{lemma:bsde_a} then the function $g = (g_b+g_d)/2$ satisfies the assumptions (ii) and (iii) in Lemma \ref{lemma:bsde_b}.
\end{remark}

For any $K\geq 0$ we consider $(Y^K, Z^K)\in \mathbb{S}^K$ the unique solution of $(\mathbf{B})_K$. We have the following convergence result for the sequence $(Y^K, Z^K)_{K\geq0}$ whose proof is given in Section \ref{proof:main_theorem}.
\begin{theorem}
\label{th:main_theorem}
Under Assumption \ref{assumption:bsde_b} if the assumptions of Lemma \ref{lemma:bsde_a} are satisfied  for any $K$ then the BSDE driven by $M^X$ with generator $g := \frac{g_b+g_d}{2}$ and terminal value $\xi$ has a unique solution $(Y, Z)$ and we have the following convergence:
$$
\Big(Y^K, \int_0^{\cdot} Z^K_t \cdot \mathrm{d}M^K_t, \langle Y^K , \overline{M}^K \rangle, \langle Y^K \rangle \Big) \rightarrow \Big(Y, \int_0^{\cdot} Z_t \mathrm{d} M_t^X, \langle Y , M \rangle, \langle Y \rangle \Big) \text{ as }K\rightarrow + \infty
$$
in $\mathcal{S}^{2}_1 \times \mathcal{S}^{2}_1 \times \mathcal{S}^{1}_2 \times \mathcal{S}^{1}_1$.
\end{theorem}

The convergence in Theorem \ref{th:main_theorem} implies the following convergence
\begin{align*}
\Big( \int_0^{\cdot} \frac{Z^{K,b}_t}{K}  \lambda^{K,b}_t \mathrm{d}t, \int_0^{\cdot} \frac{Z^{K,d}_t}{K}  \lambda^{K,d}_t \mathrm{d}t, \int_0^{\cdot} |Z^{K}_t|^2 \cdot \phi^K_t  \mathrm{d}A^K_t\Big)  \rightarrow \Big( \int_0^{\cdot} Z_t \mathrm{d}A_t/2, \int_0^{\cdot} Z_t \mathrm{d}A_t/2, \int_0^{\cdot} Z^2_t \mathrm{d}A_t  \Big),
\end{align*}
in $\mathcal{S}^{1}_1\times \mathcal{S}^{1}_1 \times \mathcal{S}^{1}_1$ when $K\rightarrow +\infty$.

\section{Application to a control problem}
\label{sec:application_control}

In this section we apply the results of Section \ref{sec:convergence} to the convergence of a sequence of controls problems.

\subsection{The discrete problem}
\label{subsec:discrete_control}

We first focus on the discrete control problem in the same spirit than Section \ref{sec:toy_model}. We consider that a resource manager monitors his harvesting intensity through a control $\alpha$, which is assumed to be bounded with bounds $\underline a,\overline a>0$. We assume that his harvesting modifies the death rate of the natural resource according to a continuous function $h^K:\mathbb R^+\times [-\underline a,\overline a]\longmapsto \mathbb R$ which satisfies the following assumption.
\begin{assumption}
\label{assumption:harvest_intensity}
There exists a positive constant $C< 2 \beta$ such that for any $(x,\alpha)\in \mathbb R^+\times [-\underline a,\overline a]$
$$
\frac{|h^K(x,\alpha)|^2}{\lambda^{K,d}(x)}\leq Cx \text{ and } Kf^{K, d}(x) + h^K(x, \alpha) \geq 0
$$
with equality if $x=0$.
\end{assumption}
The set of admissible controls is defined by
$$
\mathcal{A}^K = \{\alpha-\mathbb{F}^K~\text{predictable s.t. } \alpha\in[-\underline{a},\overline {a}]  \}.
$$

For any $\alpha\in \mathcal{A}^K$ we define 
$$
L^{K,\alpha}_t = \mathcal E\big( \int_0^{\cdot} \frac{h^K(X_s^K,\alpha_s)}{\lambda^{K,d}_s}\mathrm{d}M_s^{K,d}\big)_t,
$$
 where $\mathcal{E}$ denotes the Doleans-Dade exponential process. We deduce from Assumption \ref{assumption:harvest_intensity}, Proposition \ref{prop:expo_moments} together with \cite[Corollary 2.6]{sokol2013optimal} that $(L^{K,\alpha}_t)_{t\in [0, T]}$ is a true martingale. Hence the law of the population process under the control $\alpha$ is given by $\mathbb{P}^{K, \alpha}$ characterized by
$$
\frac{\mathrm{d}\mathbb{P}^{K, \alpha}}{\mathrm{d}\mathbb{P}} = L^{K,\alpha}_T.
$$
Under the probability $\mathbb{P}^{K, \alpha}$ the death intensity of the population becomes
$$
\lambda^{K, d, \alpha}_t = \lambda^{K, d}_t + h^K(X^K_t, \alpha_t)
$$
and the birth intensity is unchanged.\\

We assume that the manager receives at maturity $T$ a lump sum random compensation $\xi^K\in \mathbb T^K$  for his action. In addition, the manager receives continuous incomes along the time depending on the size of the population and on his control that is given by a function $c^K$ from $\mathbb R^+\times [-\underline a, \overline a]$ into $\mathbb R$. This gain can be negative which corresponds to a cost related to the effort of the manager. This is what we have considered in Section \ref{sec:toy_model}. Therefore, the goal of the manager is to solve the following maximization problem
$$
\textbf{(P)}_K:~V_0^{K} = \underset{\alpha\in \mathcal{A}^K}{\sup} J_0^{K,\alpha} \text{ with }J_0^{K,\alpha}:= \mathbb{E}^{K, \alpha}[ \xi^K + \int_0^T c^K(X^K_s, \alpha_s)\mathrm{d}s ],
$$
where $\mathbb{E}^{K, \alpha}$ denotes the expectation taken under the probability $\mathbb{P}^{K, \alpha}$. Using the notations of Section \ref{subsec:convergence_BSDEs} the BSDE associated to this control problem is
$$
\textbf{(BSDE)}_K:~ Y^K_t = \xi^K + \int_t^T g^K(X^K_s, Z^K_s)\cdot \phi^K_s \mathrm{d}A^K_s - \int_t^T Z^K_s \cdot \mathrm{d}M^K_s
$$
with for any $(x,z)\in \mathbb R^+\times \mathbb R^2$ 
\[g^K(x,z) = (0,g^{K,d}(x,z^d))\; \text{ where } g^{K,d}(x, z^d) =  \sup_{\alpha\in [-\underline a,\overline a]}\Big( c^K\big(x, \alpha\big) + z^d h^K\big(x, \alpha\big) \Big) \frac{1}{\lambda^{K, d}(x)}.\]
We need to assume that the functions $c^K$ and $h^K$ are chosen such that $g^K$ satisfies  the assumptions (ii) and (iii) of Lemma \ref{lemma:bsde_a} and that the maximizer in the above equation is unique. Formally we make the following assumption.
\begin{assumption}
\label{assumption:cv_control_a}~
\begin{itemize}
\item[(i)] $ g^{K,d}(X^K, 0)  \in \mathbb{H}^K_1 $,
\item[(ii)] there exists a positive $L$ satisfying $\beta>L^2 + 2L$ and such that for any $x,z, z',\alpha$ and $K\geq 0$ we have
$$
\frac{|K zh^K(x, \alpha)-K z'h^K(x, \alpha)|}{\lambda^{K,d}(x)}\leq L |z-z'|.
$$
\item[(iii)] For any $x,z$ there exists a unique $\alpha^{K,*}(x, z)$ such that
$$
g^{K,d}(x, z) = \frac{c^K\big( x, \alpha^{K,*}(x, z)\big) + z h^K(x, \alpha^{K,*}(x, z))}{\lambda^{K,d}(x)}.
$$
\end{itemize} 
\end{assumption}

 We thus have the following characterization of the optimal control (we refer to Appendix \ref{proof:pbdiscret} for the proof).

\begin{theorem}[Verification for $\textbf{(P)}_K$]
\label{thm:pbdiscret}
 Let $(Y^K,Z^K)\in \mathbb S^K$ be the unique solution of $\textbf{(BSDE)}_K$. Then $V_0^K=Y_0^K$ and  $\alpha^{K,\star}_t:=\alpha^{K,\star}(X_t^K,Z_t^{K,d})$ solves the problem $\textbf{(P)}_K$.
\end{theorem}

We now define the continuous version of this control problem.

\subsection{The continuous problem}
\label{subsec:continuous_control}

As previously the resource manager monitors his harvesting intensity through a control $\alpha$, assumed to be bounded with bounds $\underline a,\overline a>0$. We assume that his harvesting modified the death rate of the natural resource according to a continuous function $h:\mathbb R^+\times [-\underline a,\overline a]\longmapsto \mathbb R$ which is assumed to satisfy the following assumption.
\begin{assumption}
\label{assumption:continuous_control}
There exists a positive constant $C < 2\beta$ such that for any $(x,\alpha)\in \mathbb R^+\times [-\underline a,\overline a]$
$$
\frac{h^2(x,\alpha)}{\sigma^2 x }\leq Cx \text{ and } h(x,\alpha) + f^d(x)\geq 0.
$$
with equality if $x=0$.
\end{assumption}
The set of admissible control is
$$
\mathcal{A} = \{\alpha-\mathbb{F}^X~\text{predictable s.t. } \alpha\in[-\underline{a},\overline {a}] \}.
$$
Considering the process 
$$
L^{\alpha}_t = \mathcal E\big( \int_0^{\cdot} -\frac{ h(X_s,\alpha_s)}{\sigma^2 X_s}\mathrm{d}M_s^{X}\big)_t,
$$
where we recall that $X$ is given by $(\mathbf{S})$. We deduce from Assumption \ref{assumption:continuous_control} and Corollary 2.6 in \cite{sokol2013optimal} that $L^{\alpha}$ is a true martingale. Hence, we define the probability $\mathbb{P}^{\alpha}$ by
$$
\frac{\mathrm{d}\mathbb{P}^{\alpha}}{\mathrm{d}\mathbb{P}} =L^{\alpha}_T
$$
which is the probability measure corresponding to the control $\alpha$. Under $\mathbb{P}^{\alpha}$ the process $X$ is a strong solution of 
\[X_t = x_0 + \int_0^t \big(f(X_s)-h(X_s,\alpha_s)\big)\mathrm{d}s + \int_0^t \sigma \sqrt{X_s}\mathrm{d}B^\alpha_s,\]
where $B^\alpha:=B+\int_0^\cdot \frac{h(X_s,\alpha_s)}{\sigma \sqrt{X_s}}ds$ is a $\mathbb P^\alpha-$Brownian motion.\\

As in the discrete case we assume that the manager receives at maturity $T$ a lump sum random compensation $\xi\in \mathbb T$  for his action. In addition, the manager receives continuous incomes term depending on the size of the population and his control. This term is given by a function $c$ from $\mathbb R^+\times [-\underline a, \overline a]$ into $\mathbb R$. Therefore, the goal of the manager is to solve the following maximization problem
$$
\textbf{(P)}:~V_0 = \underset{\alpha\in \mathcal{A}}{\sup}~ J_0^{\alpha}\text{ with }J_0^{\alpha}:= \mathbb{E}^{\alpha}[ \xi + \int_0^T c(X_s, \alpha_s)\mathrm{d}s ],
$$
where $\mathbb{E}^{ \alpha}$ denotes the expectation taken under the probability $\mathbb{P}^{ \alpha}$. The BSDE associated to this control problem is
$$
\textbf{(BSDE)}:~ Y_t = \xi + \int_t^T g(X_s, Z_s)  \mathrm{d}A_s - \int_t^T Z_s  \mathrm{d}M^X_s
$$
with for any $(x,z)\in \mathbb R^+\times \mathbb R$ 
\[
g(x,z) =  \sup_{\alpha\in [-\underline a,\overline a])}\Big( c\big(x, \alpha\big) + z h\big(x, \alpha\big) \Big) \frac{1}{\sigma^{2} x}.
\]
We need to assume that the functions $c$ and $h$ are chosen such that $g$ satisfies the assumptions (ii) and (iii) of Lemma \ref{lemma:bsde_b} and that the maximizer in the above equation is unique. Formally we make the following assumption.

\begin{assumption}
\label{assumption:cv_control_cont}~
\begin{itemize}
\item[(i)]  $  g(X, 0)  \in \mathbb{H}_1 $,
\item[(ii)] there exists a positive $L$ satisfying $\beta>L^2 + 2L$ such that for any $x,z, z'$ and $\alpha$ we have
$$
\frac{| zh(x, \alpha)-z'h(x, \alpha)|}{\sigma^2 x}\leq L|z-z'|.
$$
\item[(iii)] For any $x,z$ there exists a unique $\alpha^{*}(x, z)$ such that
$$
g(x, z) = \frac{c\big( x, \alpha^{*}(x, z)\big) + z h(x, \alpha^{*}(x, z))}{\sigma^2 x}
$$
\end{itemize} 
\end{assumption}
 We thus have the following characterization of the optimal control (we refer to Appendix \ref{proof:pbcontinuous} for the proof).

\begin{theorem}[Verification for $\textbf{(P)}$]
\label{thm:pbcontinuous}
 Let $(Y,Z)\in \mathbb S$ be the unique solution of $\textbf{(BSDE)}$. Then $V_0=Y_0$ and $\alpha^{\star}_t:=\alpha^{\star}(X_t,Z_t)$ solves the problem $\textbf{(P)}$.
\end{theorem}

\subsection{Convergence of the value functions and of the optimal controls}
\label{subsec:convergence_control}

In this section, we show that under some natural assumptions the sequences of value functions $(V_0^K)_{K\geq 0}$ and of controls $(\alpha^{K,*})_{K\geq 0}$ converge respectively towards $V_0$ and $\alpha^*$. More precisely we consider the following assumptions.

\begin{assumption}
\label{assumption:control}\text{ }

\begin{itemize}
\item[(i)] $(\xi^K)_{K\geq 0}$ converges to $\xi$ in $L^{2+\varepsilon}$ for some $\varepsilon>0$,
\item[(ii)] there exists a positive sequence $(\eta_K)_{K\geq 0}$ that converges towards $0$ such that for any $x, \alpha, z$ and $K$ we have
$$
|K^2\frac{c^K(x, \alpha)}{\lambda^{K,d}(x)} - \frac{c(x, \alpha)}{\sigma^2 x /2}| + | K \frac{h^K(x, \alpha)}{\lambda^{K,d}(x)} - \frac{h(x, \alpha)}{\sigma^2 x /2}| \leq \eta_K(1 + |x| )
$$
and
$$
|\alpha^{K,*}(x, z/K) - \alpha^{*}(x, z)| + \Big| \Big(  \frac{Kh^K(x, \alpha)}{\lambda^{K,d}(x)} \Big)^2 - \Big(  \frac{h(x, \alpha)}{\sigma^2 x /2} \Big)^2 \Big| \leq \eta_K (1 + |x|+ |z|).
$$
\item[(iii)] There exists a positive constant $C>0$ such that for any $x, x', z$ and $K$ we have
$$
K^2|\frac{c^K(x, \alpha) - zK^{-1}h^K(x, \alpha)}{\lambda^{K,d}(x)} - \frac{c^K(x', \alpha) - zK^{-1}h^K(x', \alpha)}{\lambda^{K,d}(x')}| \leq C|x-x'|
$$
and for any $x, x',z,z', \alpha,\alpha'$ and $K$ we have
\begin{align*}
|\alpha^{K,*}(x, \frac{z}{K}) - \alpha^{K,*}(x', \frac{z'}{K})| + \Big| \Big(  \frac{Kh^K(x, \alpha)}{\lambda^{K,d}(x)} \Big)^2 - \Big(    \frac{Kh^K(x', \alpha')}{\lambda^{K,d}(x')} \Big)^2 \Big| \leq  C(|x-x'|+|\alpha-\alpha'| +|z-z'|).
\end{align*}
\end{itemize}
\end{assumption}

Assumption \ref{assumption:control} contains the natural assumptions ensuring that the problem $(\textbf{P})$ is the version of the problems $(\textbf{P})_K$ in the framework of the continuous population model $X$.\\

Using a slight abuse we note $\mathbb{P}^{K, *}$ the law of $X^K$ under the control $\alpha^{K,*}$ and $\mathbb{P}^{*}$ the law of $X$ under the control $\alpha^{*}$. We have the following convergence result which proof is given in Appendix \ref{proof:convergence_control}.
\begin{theorem}
\label{th:convergence_control}
\begin{itemize}
\item[(i)] We have in $\mathcal{S}^2_1\times \mathcal{S}^1_1 \times\mathcal{S}^1_1$:
$$
\lim\limits_{K\to+\infty} \Big( Y^K, \int_0^{\cdot} \alpha^{K,*}_s \lambda^{K, d}_sK^{-2}\mathrm{d}s, \int_0^{\cdot} (\alpha^{K,*}_s)^2 \lambda^{K, d}_sK^{-2}\mathrm{d}_s \Big) = \Big( Y,\int_0^{\cdot} \alpha_s  \mathrm{d}A_s/2, \int_0^{\cdot} \alpha_s^2  \mathrm{d}A_s/2 \Big).
$$ 
\item[(ii)] The sequence $(\mathbb{P}^{K, *})_{K\geq 0}$ converges for the Skorohod topology towards $\mathbb{P}^*$.
\end{itemize}
\end{theorem}

Since $Y^K_0 = V^K_0$ and $Y_0 = V_0$ a consequence of Theorem \ref{th:convergence_control} (i) is that $(V^K_0)_{K\geq 0}$ converges towards $V_0$. The point (i) also implies that the sequence of controls $(\alpha^{K})_{K\geq 0}$ converges towards the control $\alpha$. But this convergence is in a weak sense and we do not get the convergence of $(\alpha^{K,*})_{K\geq 0}$ towards $\alpha^*$ in law for the Skorohod topology.

\begin{remark}Note that the sequence of control problems considered in Section \ref{sec:toy_model} when $\alpha\in [-\nu, \overline{a}]$ (for $\overline{a}$ positive) satisfy any of the assumptions of Section \ref{sec:application_control}.
\end{remark}

\clearpage

\section{Proofs of Theorems}
\label{sec:proof}

\subsection{Proof of Theorem \ref{th:rescaling}}
\label{appendix:rescalingthm}

We introduce the process
$$
Y^K_t = \int_0^t f(X^K_s) \mathrm{d}s.
$$
The proof is divided in four main steps detailed below.
\begin{enumerate}
\item We prove that $\mathbf{(S)}$ admits a unique strong solution.
\item We show that the sequence $(Y^K, \overline{M}^{K}, \overline{N}^{K,b}, \overline{N}^{K,d}, \overline{\Lambda}^{K,b}, \overline{\Lambda}^{K,d})_{K\geq 0}$ is C-tight.
\item We show that for any limit point $(Y, M, N^b, N^d, \Lambda^b, \Lambda^d)$ of the above sequence we have $\Lambda^b = N^b=\Lambda^d = N^d$ and $Y$ is almost surely differentiable with derivative $X$ weak solution of $(\mathbf{S})$.
\item Finally, we prove that up to a probability space $(\Omega, \mathbb{F}, \mathbb{P})$ such that the above convergence holds in probability, the process  
$
(X^K, \overline{M}^{K}, \overline{N}^{K,b}, \overline{N}^{K,d}, \overline{\Lambda}^{K,b}, \overline{\Lambda}^{K,d} )_{K\geq 0}$ actually converges to $(X, M, A, A, A, A)$ when $K$ goes to $+\infty$ in $\mathcal{S}^2_1 \times \mathcal{S}^2_2 \times \mathcal{S}^1_1 \times \mathcal{S}^1_1 \times \mathcal{S}^1_1 \times \mathcal{S}^1_1$.
\end{enumerate}
\textit{Step 1: Pathwise uniqueness under existence.}
The uniqueness result is a direct consequence of \cite[IX-Theorem 3.5 (ii)]{revuz2013continuous} under Assumption $\textbf{(A)}(ii)-(iii)$. \\

\textit{Step 2: Tightness property.}

In order to show tightness we first show that the sequence $\big(\underset{t\in [0, 1] }{\sup}\mathbb{E}[X^K_t]\big)_{K\geq 0}$ is bounded uniformly with respect to $K$. We have
$$
\mathbb{E}[X^K_t] = n_0 +  \int_0^t \mathbb{E}[f^b(X^K_s) - f^d(X^K_s)]\mathrm{d}s.
$$
Hence, according to $\textbf{(A)}-(iii)$, there exists a positive constant $C$ (independent of $K$) such that
$$
\mathbb{E}[X^K_t] \leq n_0 + \int_0^t C \mathbb{E}[X^K_s] \mathrm{d}s.
$$
By using Gr\"onwall's inequality we deduce that $\big(\underset{t\in [0, 1]}{\sup}\mathbb{E}[X^K_t]\big)_{K\geq 0}$ is bounded uniformly with respect to $K$.\\

We have
$$
 \mathbb{E}[\overline{N}^{K,b}_t] \leq \int_0^t \mathbb{E}\big[  f^b(X^K_t)/K + \frac{\sigma^2(X^K_t)}{2}  \big] \mathrm{d}s,
$$
therefore $\big(\mathbb{E}[ \overline{N}^{K,b}_T] \big)_{K\geq 0}$ is bounded and since $\overline{N}^{K,d} \leq \overline{N}^{K,b}$ then $\big(\mathbb{E}[ \overline{N}^{K,d}_1] \big)_{K\geq 0}$ is also bounded. Moreover since $f(X^K) \leq C X^K$ the sequence $\big(\mathbb{E}[ Y^K_T] \big)_{K\geq 0}$ bounded. Using Theorem VI-3.21 in \cite{jacod2013limit} and that the processes $Y^K$, $\overline{N}^{K,i}$ and $\overline{\Lambda}^{K,i}$ for $i\in\{b,d\}$ are nondecreasing for any $K$ we get that the sequences $(Y^K)_{K\geq 0}$, $(\overline{N}^{K,b})_{K\geq 0}$, $(\overline{N}^{K,d})_{K\geq 0}$, $(\overline{\Lambda}^{K,b})_{K\geq 0}$ and $(\overline{\Lambda}^{K,d})_{K\geq 0}$ are tight.\\

Moreover since $|\Delta \overline{N}^{K,i}| = 1/K^2$ for $i\in\{b,d\}$ and the processes $Y^K$, $\Lambda^{K,b}$ and $\Lambda^{K,d}$ are continuous for any $K$ following Proposition VI-3.26 in \cite{jacod2013limit} we get that the sequences $(Y^K)_{K\geq 0}$, $(\overline{N}^{K,b})_{K\geq 0}$, $(\overline{N}^{K,d})_{K\geq 0}$, $(\overline{\Lambda}^{K,b})_{K\geq 0}$ and $(\overline{\Lambda}^{K,d})_{K\geq 0}$ C-are tight. \\

The tightness of $(\overline{M}^{K,b}, \overline{M}^{K,d})_{K\geq 0}$ then follows from Theorem VI-4.13 in \cite{jacod2013limit} since $\langle \overline{M}^{K,i} \rangle = \overline{\Lambda}^{K,i}$. We then get C-tightness since $|\Delta \overline{M}^{K,i}| \leq K^{-1}$. Since marginal tightness implies tightness (Corollary IV-3.33 in \cite{jacod1975multivariate}) we get that $(Y^K, \overline{M}^{K}, \overline{N}^{K,b}, \overline{N}^{K,d}, \overline{\Lambda}^{K,b}, \overline{\Lambda}^{K,d})_{K\geq 0}$ is C-tight.

\textit{Step 3: convergence of the processes and existence of a solution to $(\mathbf{S})$}

We first show the following lemma:
\begin{lemma}
\label{lemma:ucp}
For $i\in\{b,d\}$ the process $ |\overline{N}^{K,i} - \overline{\Lambda}^{K,i}|$ converges uniformly towards $0$ in probability.
\end{lemma}
\begin{proof}[Proof of Lemma \ref{lemma:ucp}]
Obviously we have  $ \overline{N}^{K,i}_t -  \overline{\Lambda}^{K,i}_t = \overline{M}^{K,i}_t / K$ so using the BDG inequality we have
$$
\mathbb{E}\big[ \underset{t\in [0, T]}{\sup} \frac{|\overline{M}^{K,i}_t|^2}{K^2} \big] \leq \frac{\mathbb{E}[\overline{N}^{K,i}_T]}{K^2}.
$$
that converges towards $0$. We conclude using Markov inequality.
\end{proof}

In view of the tightness result obtained in Step 1, we denote by $(Y, M, N^b, N^d, \Lambda^b, \Lambda^d)$ a limit point of $(Y^K, \overline{M}^{K}, \overline{N}^{K,b}, \overline{N}^{K,d}, \overline{\Lambda}^{K,b}, \overline{\Lambda}^{K,d})_{K\geq 0}$ with $M = (M^b,M^d)$. By the Skorohod representation theorem since the limit of each marginal is continuous we can consider that $(Y^K, \overline{M}^{K}, \overline{N}^{K,b}, \overline{N}^{K,d})_{K\geq 0}$ converges almost surely and uniformly on $[0, T]$ towards $(Y, M, N^b, N^d)$, \textit{i.e.}
$$ 
\underset{t \in [0, 1]}{\sup }|Y^{ K}_t - Y_t| \underset{K \rightarrow + \infty}{\rightarrow} 0,
$$
and for any $i\in\{b,d\}$
$$
\underset{t \in [0, 1]}{\sup }|\overline{M}^{K,i}_t - M^i_t| \underset{K \rightarrow + \infty}{\rightarrow} 0,\quad\underset{t \in [0, 1]}{\sup }|\overline{N}^{K,i}_t - N^i_t| \underset{K \rightarrow + \infty}{\rightarrow} 0,\quad \underset{t \in [0, 1]}{\sup }|\overline{\Lambda}^{K,i}_t - \Lambda^i_t| \underset{K \rightarrow + \infty}{\rightarrow} 0.
$$
According to Corollary IX-1.19 in \cite{jacod2013limit} we have that $M$ is a local martingales. Moreover we have $[\overline{M}^{K,i}] = \overline{N}^{K,i} $ and $[\overline{M}^{K,b}, \overline{M}^{K,d}] = 0$ so Corollary VI-6.29 in \cite{jacod2013limit} gives $[M^i] = N^i = \Lambda^i$ and $[M^b, M^d] = 0$. Since $M^i$ is a continuous martingale we get
$\langle M^i \rangle = [M^i] = \Lambda^i.$
We also notice that $\mathbb{E}[\overline{\Lambda}^{K,i}_T]$ is uniformly bounded in $K$, so according to Fatou's lemma $\Lambda^i$ is integrable, therefore $M^i$ is a true martingale.\\

We recall that
$$
X^K_t = n_0 + \int_0^t f(X^K_s)\mathrm{d}s + \overline{M}^{K,b}_t - \overline{M}^{K,d}_t.
$$
Then, $X^K$ converges almost surely and uniformly on $[0, T]$ towards 
$$
X_t := n_0 + Y_t + M^b_t - M^d_t
$$
and $Y^K$ converges almost surely uniformly on $[0, T]$ towards
$$
\int_0^{\cdot} f(X_s)\mathrm{d}s.
$$
Since we have
$$
\langle M^b \rangle_t = \langle M^d \rangle_t = \int_0^t \frac{\sigma^2(X_s)}{2} \mathrm{d}s \text{ and }\langle M^b, M^d \rangle = 0
$$
we get from Theorem V-3.9 in \cite{revuz2013continuous} that there exists a bi-dimensional Brownian motion $(B^b, B^d)$ such that
$$
(M^b_t, M^d_t) = \int_0^t \frac{\sigma(X_s)}{\sqrt{2}}\mathrm{d}(B^b_s, B^d_s).
$$
So finally we have shown that 
\begin{equation*}
X_t = n_0 + \int_0^t f(X_s) \mathrm{d}s + \int_0^t \sigma (X_s)\mathrm{d}(\frac{B^b_s + B^d_s}{\sqrt{2}}).
\end{equation*}
This concludes the proof of the first part of Theorem \ref{th:rescaling} since $
(X^K, \overline{M}^{K}, \overline{N}^{K,b}, \overline{N}^{K,d}, \overline{\Lambda}^{K,b}, \overline{\Lambda}^{K,d} )$ converges in law for the Skorohod topology to $(X,M,A,A,A,A)$.\\

\textit{Step 4: convergence of a copy $
(X^K, \overline{M}^{K}, \overline{N}^{K,b}, \overline{N}^{K,d}, \overline{\Lambda}^{K,b}, \overline{\Lambda}^{K,d} )_{K\geq 0}$ in $\mathcal{S}^2_1 \times \mathcal{S}^2_2 \times \mathcal{S}^1_1 \times \mathcal{S}^1_1\times \mathcal{S}^1_1 \times \mathcal{S}^1_1$. }

In view of the conclusion of Step 3, and by the Skorohod representation theorem, there exists a probability space $(\Omega, \mathbb{F}, \mathbb{P})$ and a copy in law of the sequence of processes  
$
(X^K, \overline{M}^{K}, \overline{N}^{K,b}, \overline{N}^{K,d}, \overline{\Lambda}^{K,b}, \overline{\Lambda}^{K,d} )_{K\geq 0}$ that converges in probability toward a copy of $(X, M, A, A, A, A)$ when $K$ goes to $+\infty$. To prove that the convergence actually holds in $\mathcal{S}^2_1 \times \mathcal{S}^2_2 \times \mathcal{S}^1_1 \times \mathcal{S}^1_1 \times \mathcal{S}^1_1 \times \mathcal{S}^1_1$ we show that:
\begin{itemize}
\item[$(i)$] $(\overline{N}^{K, b})_{K\geq 0}$ and $(\overline{N}^{K, b})_{K\geq 0}$ are bounded in $\mathcal{S}^2_1$,
\item[$(ii)$] $(\overline{\Lambda}^{K, b})_{K\geq 0}$ and $(\overline{\Lambda}^{K, d})_{K\geq 0}$ are bounded in $\mathcal{S}^2_1$,
\item[$(iii)$] $(\overline{M}^K)_{K\geq 0}$ is bounded in $\mathcal{S}^4$,
\item[$(iv)$] $(X^{K})_{K\geq 0}$ is bounded in $\mathcal{S}^4_1$.
\end{itemize}
Then we will get the convergence using dominated convergence.

\textit{Proof of $(i)$.}
We write
$$
\underset{t\in [0, T]}{\sup}(\overline{N}^{K, b}_s)^2 =  (\overline{N}^{K, b}_T)^2 = \int_0^T (2\frac{\overline{N}_s^{K, b}}{K^2}+K^{-4})\mathrm{d}N^{K, b}_s.
$$
Therefore we have for a positive constant $C$ independent of $K$ such that
$$
\mathbb{E}[\underset{t\in [0, T]}{\sup}(\overline{N}^{K, b}_s)^2] = \mathbb{E}[\int_0^T (2\overline{N}_s^{K, b}+K^{-2})\lambda^{K, b}_s\mathrm{d}s] \leq \mathbb{E}[\int_0^T (2\overline{N}_s^{K, b}+K^{-2})C X^K_s\mathrm{d}s].
$$
Hence to conclude it is enough to show that  $(\mathbb{E}[\overline{N}^{K, b}_t X^K_t])_{t\in [0, T]}$ is bounded. We have
\begin{align*}
\overline{N}^{K, b}_t X^K_t &= \int_0^t (X^K_sK^{-2}+ \overline{N}^{K, b}_sK^{-1} +K^{-3}) \mathrm{d}N^{K, b}_s - \overline{N}^{K, b}_s K^{-1} \mathrm{d}N^{K, d}_s\\
& = \int_0^t (X^K_sK^{-2} +K^{-3}) \mathrm{d}N^{K, b}_s + \int_0^t \overline{N}^{K, b}_s K^{-1} \mathrm{d}(N^{K, b}_s-N^{K, d}_s).  
\end{align*}
So we get
\begin{align*}
\mathbb{E}[\overline{N}^{K, b}_t X^K_t] &= \mathbb{E}[ \int_0^t (X^K_s  + K^{-1}) K^{-2} \lambda^{K, b}_s \mathrm{d}s + \int_0^t \overline{N}^{K, b}_s K^{-1}( \lambda^{K, b}_s - \lambda^{K, d}_s ) \mathrm{d}s]\\
&\leq \mathbb{E}[ \int_0^t (X^K_s  + K^{-1}) C X^K_s \mathrm{d}s + \int_0^t \overline{N}^{K, b}_s C X^{K}_s \mathrm{d}s].
\end{align*}
Therefore by Proposition \ref{prop:expo_moments} and Gr\"onwall lemma we get point (i) (since the same results follows for $N^{K, d}$).\\

\textit{Proof of $(ii)$.} We have
$$
\underset{t\in [0, T]}{\sup}\overline{\Lambda}^{K, d}_t \leq C \int_0^T X^K_s \mathrm{d}s,
$$
therefore point $(ii)$ follows from Proposition \ref{prop:expo_moments}. Same proof holds for $\overline{N}^{K, d}$.\\

\textit{Proof of $(iii)$. }Using the Burkholder-Davis-Gundy inequality we get
$$
\mathbb{E}[\underset{t\in [0, T]}{\sup}|\overline{M}^{K, b}_t|^4] \leq C \mathbb{E}[(\overline{N}^{K, b}_T)^2].
$$
Therefore because of point $(i)$ we get point $(iii)$. Same proof holds for $\overline{M}^{K, d}$.

\textit{Proof of $(iv)$.} We write
\begin{align*}
X^K_t &= x_0^K + \int_0^{t} (\lambda^{K, b}_s-\lambda^{K, d}_s)K^{-1} \mathrm{d}s + \overline{M}^{K, b}_t - \overline{M}^{K, d}_t \\
&\leq x_0^K + \int_0^{t} C X^K_s \mathrm{d}s + \overline{M}^{K, b}_t - \overline{M}^{K, d}_t.
\end{align*}
So we have
$$
| X^K_t |^4 \leq C\Big( (x_0^K)^4 + \big( \int_0^{T} X^K_s \mathrm{d}s \big)^4 + |\overline{M}^{K, b}_t|^4 + |\overline{M}^{K, d}_t|^4 \Big)
$$
taking the supremum over $t\in [0, T]$ and then the expectation we obtain point $(iv)$ as corollary of Proposition \ref{prop:expo_moments} and point $(iii)$.

\subsection{Proof of Theorem \ref{th:main_theorem}.}
\label{proof:main_theorem}

The proof of Theorem \ref{th:main_theorem} is inspired from the proof of Theorem 12 in \cite{briand2002robustness}.\\

We proceed in $3$ steps:
\begin{itemize}
\item[(i)] We show that there exists $\alpha\in (0, 1)$ and some $\alpha-$contracting functions $(F^K)_{K\geq 0}$ and $F$ such that for any $K$, the unique solution of $(\textbf{B})_K$ is the fixed point of $F^K$ and the fixed point of $F$ is solution to $(\mathbf{B})$.
\item[(ii)] We introduce a double indexed sequence and prove a convergence result by induction.
\item[(iii)] We conclude.
\end{itemize}

\subsubsection{Step $(i)$}
\label{proof:contracting_functions}

For any $K$ we define the function 
$$
\begin{array}{l|rcl}
F^K : & \mathbb{S}^{K} & \longrightarrow & \mathbb{S}^K \\
    & (Y, Z) & \longmapsto & (Y', Z') \end{array}
$$
where $(Y', Z')$ is the unique solution of the BSDE:
$$
Y'_t = \xi^K + \int_t^T g^K(X^K_s, Y_s, Z_s) \cdot \phi^K_s \mathrm{d}A^K_s - \int_t^T Z'_s \cdot \mathrm{d} M^K_s.
$$
Since $(Y, Z)\in \mathbb{S}^K$ and because of assumptions (ii) and (iii) in Lemma \ref{lemma:bsde_a} we have $g^K(X^K, Y, Z) \in \mathbb{H}_2^K$. So we can properly define
$$
Y'_t = \mathbb{E}[ \xi^K + \int_t^T g^K(X^K_s, Y_s, Z_s) \cdot \phi^K_s \mathrm{d}A^K_s |\mathcal{F}^K_t ]
$$
and $Z'$ is the unique process in $\mathbb{H}_2^K$ satisfying
$$
\mathbb{E}[\xi^K + \int_0^T g^K(X^K_s, Y_s, Z_s) \cdot \phi^K_s \mathrm{d}A^K_s |\mathcal{F}^K_t ] = {Y_0'}^K + \int_0^t Z'_s \cdot \mathrm{d} M^K_s.
$$

Consider two pairs $(Y^1, Z^1)$, $(Y^2, Z^2)\in \mathbb{S}^K$ and noting $(\overline{Y}^1, \overline{Z}^1) = F^K(Y^1, Z^1)$ (resp. $(\overline{Y}^2, \overline{Z}^2) = F^K(Y^2, Z^2)$). Using Ito's formula on $e^{\beta A^K_t}|\overline{Y}^1_t - \overline{Y}^2_t|^2$ between $0$ and $T$ we get
\begin{align*}
 - |\overline{Y}^1_0 - \overline{Y}^2_0|^2 =& \int_0^T e^{\beta A^K_t}\Big( \beta |\overline{Y}^1_t - \overline{Y}^2_t|^2 -  2(\overline{Y}^1_t - \overline{Y}^2_t) \big( g_K(X^K_t, Y^1_t, Z^1_t)-g_K(X^K_t, Y^2_t, Z^2_t)\big) \cdot \phi^K_t \Big) \mathrm{d}A^K_t \\
&+ \int_0^T e^{\beta A^K_t}2(\overline{Y}^1_t - \overline{Y}^2_t)(\overline{Z}^{1}_t - \overline{Z}^{2}_t)\cdot \mathrm{d}M^K_t\\
&+ \int_0^T e^{\beta A^K_t} |\overline{Z}^{1,b}_t - \overline{Z}^{2,b}_t |^2 \mathrm{d}N^b_t + \int_0^T e^{\beta A^K_t} |\overline{Z}^{1,d}_t - \overline{Z}^{2,d}_t |^2 \mathrm{d}N^d_t.
\end{align*}
Taking the expectation we get
\begin{align*}
 |\overline{Y}^1_0 - \overline{Y}^2_0|^2 &+ \mathbb{E}[\int_0^T e^{\beta A^K_t} \beta |\overline{Y}^1_t - \overline{Y}^2_t|^2\mathrm{d}A^K_t ]  + \mathbb{E}[\int_0^T e^{\beta A^K_t} |\overline{Z}^{1}_t - \overline{Z}^{2}_t |^2 \cdot \phi^K_t \mathrm{d}A^K_t]\\
 & \leq \mathbb{E}[\int_0^T e^{\beta A^K_t }  2|\overline{Y}^1_t - \overline{Y}^2_t| \big| g^K(X^K_t, Y^1_t, Z^1_t)-g^K(X^K_t, Y^2_t, Z^2_t)\big| \cdot \phi^K_t  \mathrm{d}A^K_t].
\end{align*}
Therefore using the assumptions of Lemma \ref{lemma:bsde_a} together with Young's inequality we get for any positive $\alpha$ and $\gamma$ that
$$
\beta \|\overline{Y}^1 - \overline{Y}^2\|_{\mathbb{H}^K_1}^2 + \|\overline{Z}^1 - \overline{Z}^2\|_{\mathbb{H}^K_2}^2 \leq ( \frac{L}{\gamma} + \frac{L}{\alpha}) \|\overline{Y}^1-\overline{Y}^2\|^2_{\mathbb{H}^K_1}  + L \gamma \|Z^1-Z^2\|^2_{\mathbb{H}^K_2} + L\alpha \|Y^1-Y^2\|^2_{\mathbb{H}^K_1}
$$
or equivalently
$$
(\beta - \frac{L}{\alpha} - \frac{L}{\gamma} ) \|\overline{Y}^1 - \overline{Y}^2\|_{\mathbb{H}^K_1}^2 + \|\overline{Z}^1 - \overline{Z}^2\|_{\mathbb{H}^K_2}^2 \leq  L\gamma \|Z^1-Z^2\|^2_{\mathbb{H}^K_2}  +  L\alpha \|Y^1-Y^2\|^2_{\mathbb{H}^K_1}.
$$
Inspired by the proof of Theorem 3.4 in \cite{confortola2013backward} we choose $\gamma = \alpha/L $ and $\alpha\in (0, 1)$ such that
$$
\beta - \frac{L^2 + L}{\alpha} > L.
$$
We can make such choice since $\beta - L^2 -L>L$. Therefore we obtain
$$
L\|\overline{Y}^1 - \overline{Y}^2\|_{\mathbb{H}^K_1}^2 + \|\overline{Z}^1 - \overline{Z}^2\|_{\mathbb{H}^K_2}^2 \leq \alpha \big( L \|Y^1 - Y^2\|_{\mathbb{H}^K_1}^2 + \|Z^1 - Z^2\|_{\mathbb{H}^K_2}^2 \big).
$$
Therefore for any $K$ the function $F^K$ is an $\alpha-$contraction on $\mathbb{S}^K$ for the norm equivalent to $\|\cdot\|_{\mathbb{S}^K}$ and defined by
$$
\|(Y, Z)\|'_{\mathbb{S}^K} = \Big( L \|Y\|_{\mathbb{H}^K_1}^2 + \|Z\|_{\mathbb{H}^K_2}^2\Big)^{1/2}
$$ 

In the continuous case we consider
$$
\begin{array}{l|rcl}
F : & \mathbb{S} & \longrightarrow & \mathbb{S} \\
    & (Y, Z) & \longmapsto & (Y', Z') \end{array}
$$
where $(Y', Z')$ is the unique solution of the BSDE:
$$
Y'_t = \xi + \int_t^T g(X_s, Y_s, Z_s) \mathrm{d}As - \int_t^T Z'_s \cdot \mathrm{d} M^X_s.
$$
Since $(Y, Z)\in \mathbb{S}$ and because of Remark \ref{remark:bsde} we have $g(X, Y, Z) \in \mathbb{H}_2$. So we can properly define
$$
Y'_t = \mathbb{E}[ \xi + \int_t^T g(X_s, Y_s, Z_s) \cdot \phi_s \mathrm{d}As |\mathcal{F}_t ]
$$
and $Z'$ is the unique process in $\mathbb{H}_2$ satisfying
$$
\mathbb{E}[\xi + \int_t^T g(X_s, Y_s, Z_s) \cdot \phi^K_s \mathrm{d}As |\mathcal{F}_t ] = \int_0^t Z'_s \cdot \mathrm{d} M_s.
$$

Similarly we obtain that $F$ is an $\alpha-$contraction for the equivalent norm on $\mathbb{S}$:
$$
\|(Y, Z)\|'_{\mathbb{S}} = \Big( L \|Y\|_{\mathbb{H}}^2 + \|Z\|_{\mathbb{H}}^2 \Big)^{1/2}.
$$

\subsubsection{Step $(ii)$}
\label{proof:induction}

For any $K\geq 0$ we define the sequence $(Y^{K, p}, Z^{K, p})_{p\geq 0}$ satisfying 
$$
(Y^{K,0}, Z^{K, 0}) = 0 \text{ and } (Y^{K, p+1}, Z^{K, p+1}) = F^K(Y^{K, p}, Z^{K, p}).
$$
We similarly consider the sequence $(Y^p, Z^p)_{p\geq 0}$ defined by 
$$
(Y^0, Z^0) = 0 \text{ and } (Y^{p+1}, Z^{p+1}) = F(Y^p, Z^p).
$$

Since for any $K\geq 0$, $F^K$ is a contraction. For any $K\geq0$ the sequence $(Y^{K,p}, Z^{K,p})_{p\geq 0}$ converges towards $(Y^K, Z^K)$ in $\mathbb{S}^K$. In the same way $(Y^{p}, Z^{p})_{p\geq 0}$ converges towards $(Y, Z)$ in $\mathbb{S}$.\\

We use the following notation:
$$
Q^{K, p+1}_t = \int_0^t Z^{K, p+1}_s \cdot \mathrm{d}M^K_s \text{, } \chi^{K, p}_t:= \int_0^t g_K(X^K_s, Y^{K, p}_s, Z^{K, p}_s)\cdot\phi_s^K \mathrm{d}A^K_s,
$$ 
$$
Q^{p+1}_t = \int_0^t Z^{p+1}_s  \mathrm{d}M^X_s \text{ and }\chi^p_t:=\int_0^t g(X_s, Y^p_s, Z^p_s)\mathrm{d}A_s.
$$
So that we can write:
\begin{equation}
\label{eq:writing_y_k}
Y^{K, p+1}_t = \xi^K + \chi^{K, p}_T - \chi^{K, p}_t - Q^{K, p+1}_T + Q^{K, p+1}_t
\end{equation}
and
\begin{equation}
\label{eq:writing_y}
Y^{p+1}_t = \xi + \chi^p_T - \chi^p_t - Q^{p+1}_T + Q^{p+1}_t.
\end{equation}

We prove by induction that the following convergence holds for any $p$:
$$
\Big( Y^{K, p}, Q^{K, p}, \langle Q^{K, p},\overline{M}^{K}\rangle, \langle Q^{K, p}\rangle \Big) \rightarrow \Big( Y^{p}, Q^p, \langle Q^{p},M\rangle, \langle Q^{p}\rangle \Big)\text{ as }K\rightarrow +\infty 
$$
in $\mathcal{S}^{2+\varepsilon_p}_1 \times \mathcal{S}^{2+\varepsilon_p}_1 \times \mathcal{S}^{2+\varepsilon_p}_2 \times \mathcal{S}^{1+\varepsilon_p/2}_1 $ where $\varepsilon_p = \varepsilon/2^p $.\\

Obviously the result holds for $p=0$. We assume that the converge holds for $p$ and show that it implies the convergence for $p+1$.\\

We write
$$
\mathbb{E}[\xi^K + \chi^{K, p}_T |\mathcal{F}^K_t] = Y_0^{K, p+1} +Q^{K, p+1}_t\text{ and }\mathbb{E}[\xi+\chi^{p}_T|\mathcal{F}^X_t] = Y_0^{p+1} + Q^{p+1}_t.
$$
We prove in Appendix \ref{proof:convergence_integral} that the induction hypothesis implies that $(\chi^{K, p})_{K\geq 0}$ converges towards $\chi^p$ in $\mathcal{S}_1^{2 + \tilde{\varepsilon}_{p}}$ where $\tilde{\varepsilon}_{p} = (\varepsilon_{p} + \varepsilon_{p+1})/2$. Therefore $(\xi^K + \chi^{p, K}_T)_{K\geq 0}$ converges towards $(\xi + \chi^p)$ in $L^{2+\tilde{\varepsilon}_{p}}$. Since $\tilde{\varepsilon}_{p}>\varepsilon_{p+1}$ using Proposition \ref{prop:convergence_representation} we get
$$
\Big( Q^{K, p+1}, \langle Q^{K, p+1}, \overline{M}^K \rangle, \langle Q^{K, p+1} \rangle  \Big) \rightarrow \Big( Q^{p+1}, \langle Q^{p+1}, M \rangle,  \langle Q^{p+1}\rangle \Big)
$$
in $\mathcal{S}^{2+\varepsilon_{p+1}}_1 \times\mathcal{S}^{2+\varepsilon_{p+1}}_2 \times \mathcal{S}^{1+\varepsilon_{p+1}/2}_1$. From equations \eqref{eq:writing_y_k} and \eqref{eq:writing_y} we immediatly get that $(Y^{K, p+1})_{K\geq 0}$ converges towards $Y^p$ in $\mathcal{S}^{2+\varepsilon_{p+1}}_1$. Therefore we get the convergence for $p+1$.

\subsubsection{Step $(iii)$}
\label{proof:conclusion}

Note that a consequence of Step $(i)$ is that for a certain positive constant $C$ we have
\begin{equation}
\label{eq:ineq_contraction}
\|(Y^{K, p}, Z^{K, p}) - (Y^{K}, Z^{K})\|_{\mathbb{S}^K} + \|(Y^{p}, Z^{p}) - (Y, Z)\|_{\mathbb{S}}\leq C \alpha^p.
\end{equation}

We write
$$
\| Q^K - Q \|_{2} \leq  \|Q^p-Q\|_{2} + \|Q^K-Q^{K,p}\|_{2} + \|Q^{K,p}-Q^p\|_{2}.
$$
Notice that according to the BDG inequality there exists a positive constant $C$ such that for any $K$
$$
\|Q^{K, p} -Q^K\|^2_2 + \|Q^{p} -Q \|^2_2 \leq C \big( \|Z^{K, p} - Z^K\|^2_{\mathbb{H}^K}  + \|Z^{p} - Z\|^2_{\mathbb{H}} \big) 
$$
which converges towards $0$ uniformly in $K$ when $p\rightarrow +\infty$ by Equation \eqref{eq:ineq_contraction}. Hence $(Q^K)_{K\geq 0}$ converges in $\mathcal{S}^2_1$ towards $Q$.\\

Similarly we write
$$
\|Y^K-Y\|_{2} \leq \|Y^p-Y\|_{2}  + \|Y^{K, p}-Y^K\|_{2} + \|Y^{K,p}-Y^p\|_{2}.
$$
We proved in the previous section that the last term goes to $0$ when $K\rightarrow + \infty$. Remark that we have
$$
Y^K_t - Y_t^{K, p} = \mathbb{E}[\int_t^T \big( g^K(X^K_s, Y^K_s, Z^K_s)- g^K(X^K_s, Y^{K,p-1}_s, Z^{K, p-1}_s)\big) \cdot \phi^{K}_s \mathrm{d}A^K_s|\mathcal{F}^K_t ]
$$
so using Jensen and Doob's inequality we get
\begin{align*}
\mathbb{E}[\underset{t\in [0, T]}{\sup} e^{\beta A^K_t} |Y^K_t - Y_t^{K, p-1}|^2]& \leq L^2 \mathbb{E}[\int_0^Te^{\beta A^K_s} |Y^K_s-Y^{K,p-1}_s|^2 \mathrm{d}A^K_s +\int_0^T  e^{\beta A^K_s} |Z^K_s-Z^{K,p-1}_s|^2 \cdot \phi^K_s \mathrm{d}A^K_s]\\
& \leq L \|(Y^{K,p-1}, Z^{K,p-1}) - (Y^K, Z^K)\|^2_{\mathbb{S}^K}.
\end{align*}
Therefore $\| Y^K - Y^{K,p} \|_{\mathbb{K}^K}$ goes to $0$ when $p\rightarrow + \infty$. In the same way we get that $\| Y - Y^{p} \|_{\mathbb{K}}$ goes to $0$ when $p\rightarrow + \infty$. So $\|Y^K-Y\|_{2}$ converges towards $0$.\\

Finally notice that,
$$
\langle Y^K, \overline{M}^K\rangle = \langle Q^K, \overline{M}^K\rangle \text{, }\langle Y^K \rangle = \langle Q^K\rangle,~\langle Y, ~M \rangle = \langle Q, M\rangle \text{ and }\langle Y \rangle = \langle Q\rangle.
$$
So the convergence 
$$
\big(\langle Y^K, \overline{M}^K\rangle, \langle Y^K\rangle \big)\rightarrow \big(\langle Y, M \rangle, \langle Y \rangle \big) \text{ as }K\rightarrow +\infty
$$ 
in $\mathcal{S}^1_2 \times \mathcal{S}^1_1$ follows from Proposition 2 in \cite{briand2002robustness} and from the convergence of $(Q^K, \overline{M}^K)_{K\geq 0}$ in $\mathcal{S}^{2}_1 \times \mathcal{S}^2_2$ towards $(Q, M)$.

\subsubsection{Convergence of $(\chi^{p, K})_{K\geq 0}$ towards $\chi^p$}
\label{proof:convergence_integral}

To prove the convergence we first prove that $(\chi^{p, K})_{K\geq 0}$ converges towards $\chi^p$ in probability for the uniform topology. Then we show that $(|\chi^{p, K}|+|\chi^p|)_{K\geq 0}$ is bounded in $\mathcal{S}^{2+\hat{\varepsilon}_{p}}$ where $\hat{\varepsilon}_{p} = (\varepsilon_p + \tilde{\varepsilon}_p )/2> \tilde{\varepsilon}_{p}$. We conclude by dominated convergence.\\

For any $n$ we note $Z^{p, n}=Z^p \mathbf{1}_{|Z^p|<n}$. We write
\begin{align*}
\underset{t\in [0, T]}{\sup} | \phi^K_t - \phi_t |  \leq  \sum_{i = b, d} T^{n, K, p}_{i,1} + T^{n, K, p}_{i,2} + T^{n, K, p}_{i,3} + T^{n, p}_{i,4}/2
\end{align*}
where for $i\in {b, d}$, we recall that $g = (g_b+g_d)/2$,
\begin{align*}
T^{n, K, p}_{i,1} &= \underset{t\in [0, T]}{\sup} |\int_0^t g^K_i(X^K_s, Y^{K, p}_s, Z^{K, p}_s) \phi^{K, i}_s\mathrm{d}A^K_s - \int_0^t g^K_i(X^K_s, Y^{K, p}_s, Z^{p,n}_s/K) \phi^{K, i}_s\mathrm{d}A^K_s| \\
T^{n, K, p}_{i,2} &= \underset{t\in [0, T]}{\sup} |\int_0^t g^K_i(X^K_s, Y^{K, p}_s, Z^{p,n}_s/K) \phi^{K, i}_s\mathrm{d}A^K_s - \int_0^t g_i(X_s, Y^p_s, Z^{p,n}_s) K^{-2}\phi^{K, i}_s\mathrm{d}A^K_s| \\
T^{n, K, p}_{i,3} &= \underset{t\in [0, T]}{\sup} |\int_0^t g_i(X_s, Y^{p}_s, Z^{p,n}_s) K^{-2} \phi^{K, i}_s\mathrm{d}A^K_s - \int_0^t g_i(X_s, Y^{p}_s, Z^{p,n}_s) /2 \mathrm{d}A_s | \\
T^{n, p}_{i,4} &= \underset{t\in [0, T]}{\sup} |\int_0^t g_i(X_s, Y^{p}_s, Z^{p, n}_s) -  g_i(X_s, Y^p_s, Z^{p}_s) \mathrm{d}A_s|.
\end{align*}
For $i\in\{b,d\}$ the sequence $(T^{n,p}_{i, 4})_{n\geq 0}$ obviously converges to $0$ in probability by almost sure convergence as $n$ goes to infinity.\\

The sequence $(\int_0^{\cdot}K^{-2}\phi^{K, i}_s\mathrm{d}A^K_s)_{K\geq 0}$ converges towards $A/2$ in probability for the Skorohod topology and satisfy the P-UT condition by Proposition VI-6.12 in \cite{jacod2013limit}. So for any $n$, $(T^{n,K,p}_{i,3})_{K\geq 0}$ converges towards $0$ in probability (for the Skorohod topology) as a consequence of Theorem VI-6.22 in \cite{jacod2013limit}.\\

For the second term we write
\begin{align*}
|K^2 g^K_b(X^K, Y^{K, p}, Z^{p, n}/K) &- g_b(X, Y^{p}, Z^{p, n})|\\
&\leq  | K^2g^K_b(X^K, Y^{K, p}, Z^{p, n}/K) - K^2 g^K_b(X, Y^{p}, Z^{p, n}/K)| \\
&~~+ |K^2 g^K_b(X, Y^{p}, Z^{p, n}/K) - g_b(X, Y^{p}, Z^{p, n})|.
\end{align*}
So by assumptions of Lemma \ref{lemma:bsde_a} and Assumption \ref{assumption:bsde_b} we get:
\begin{align*}
|K^2  g^K_b(X^K, Y^{K, p}, Z^{p, n}/K)& - g_b(X, Y^{p}, Z^{p, n})|\\
\leq& C |X^K-X| + L|Y^{K, p}-Y^p|  + \upsilon_K (|X|^2 + |Y^p|^2 + n^2) .
\end{align*}
Thus there exits $\tilde{C}>0$
\begin{align*}
|T^{n,K,p}_{b,2}| + |T^{n,K,p}_{d,2}|\leq & \tilde{C} A^K_T \Big(  \underset{t\in [0, T]}{\sup} |X^K_t - X_t|  + \underset{t\in [0, T]}{\sup} |Y^{K, p}_t - Y^p_t| \\
&+  \upsilon_K (\underset{t\in [0, T]}{\sup } |X_t|^2+\underset{t\in [0, T]}{\sup } |Y^p_t|^2+n^2 ) \Big) 
\end{align*}
 which obviously goes to $0$ in probability when $K\rightarrow + \infty$ according to Slutsky's theorem, in view of the induction hypothesis and since $(\upsilon_K)_{K\geq 0}$ goes to $0$.\\

Finally we write:
$$
|K^2g^K_b(X^K_s, Y^{K, p}_s, Z^{K, p}_s) - K^2g^K_b(X^K_s, Y^{K, p}_s, Z^{p,n}_s/K)|^2 \leq L^2 |KZ^{K, p, b}_s - Z^{p,n}_s|^2.
$$
So we have
$$
|T^{n,K,p}_{b,1}| + | T^{n,K,p}_{d,1} |\leq \underset{t\in [0, T]}{\sup} L^2 \int_0^t \big( (KZ^{K, p}_s)^2 + (Z^{p,n}_s)^2 - 2Z^{p,n}_sKZ^{K, p}_s \big) \cdot K^{-2}\phi^{K}_s\mathrm{d}A^K_s
$$
Taking the average and going to the upper limit in $K$ we get by induction hypothesis and from Theorem VI-6.22 in \cite{jacod2013limit} that
\begin{align*}
\underset{K\rightarrow + \infty}{\lim\sup}~ \mathbb{E}[|T^{n,K,p}_{b,1}| + |T^{n,K,p}_{d,1} |] \leq& L^2 \mathbb{E}\big[ \underset{t\in [0, T]}{\sup} \int_0^t \big( (Z^{p}_s)^2 + (Z^{p,n}_s)^2 - 2Z^{p,n}_sZ^{p}_s \big) \mathrm{d}A_s \big]\\
 \leq & L^2 \mathbb{E}\big[ \underset{t\in [0, T]}{\sup} \int_0^t  (Z^{p}_s - Z^{p,n}_s)^2  \mathrm{d}A_s  \big].
\end{align*}
The RHS converges to $0$ when $n\rightarrow+\infty$ by dominated convergence. Hence we have shown that $(\chi^{K, p})_{K \geq 0}$ converges to $\chi^p$ in probability for the uniform convergence.\\

To conclude we show that $(|\chi^{K, p}|+|\chi^p|)_{K\geq 0}$ is bounded in $S^{2+\hat{\varepsilon}_{p}}$. We write
\begin{align*}
\underset{t\in [0, T]}{\sup}|\chi_t^{K, p}| \leq & C A^K_T \Big( 1 + \underset{t\in [0, T]}{\sup }|X^K_t| + \underset{t\in [0, T]}{\sup }|Y^K_t| \Big)  + C  \int_0^T |KZ_s^K|\cdot \frac{\phi^K_s}{K^2} \mathrm{d}A^K_s.
\end{align*}
Using Kunita-Watanabe it is easy to see that the last term is bounded in $L^{2+\hat{\varepsilon}_{p}}$. The other terms are bounded in $L^{2+\hat{\varepsilon}_{p}}$ by induction assumption and Proposition \ref{prop:expo_moments}. So $(\chi^{K, p})_{K\geq 0}$ is bounded in $\mathcal{S}_1^{2+\hat{\varepsilon}_{p}}$. In the same way we show that $\chi^p\in \mathcal{S}_1^{2+\hat{\varepsilon}_{p}}$. Therefore we obtain the convergence of $(\chi^{K, p})_{K \geq 0}$ towards $\chi^p$ in $\mathcal{S}_1^{2+ \tilde{\varepsilon}_{p}}$ by dominated convergence.

\subsection{Proof of Theorem \ref{thm:pbdiscret}}
\label{proof:pbdiscret}

From Assumption \ref{assumption:cv_control_a} (i)-(ii) we get that the generator $g^K$ satisfies the conditions of Lemma \ref{lemma:bsde_a}. Therefore $(\textbf{BSDE})_K$ admits a unique solution $(Y^K, Z^K)\in \mathbb{S}^K$. We consider $\alpha^{K,*}_t = \alpha^{K,*}(X^K_t,Z^K_t)$, and show that $\alpha^{K,*}$ solve the optimal control problem $(\mathbf{C})_K$. Since $\alpha^{K,*}$ is admissible according to Assumption \ref{assumption:cv_control_a} (iii) we have $J_0^{K, \alpha^{K,*}} = Y_0^K$.\\

We now take any $\alpha \in \mathcal{A}^K$ and show that
$$
J_0^{K, \alpha^{K*}} \geq J_0^{K, \alpha}.
$$
We write:
\begin{align*}
J_0^{K, \alpha^{K*}} = \xi^K + \int_0^T  \big( c^K(X^K_t, \alpha^{K*}_t) + Z^{K,d}_t h^K(X^K_t, \alpha^{K*}_t) - c^K(X^K_t, \alpha_t) - Z^{K,d}_t h^K(X^K_t, \alpha_t) \big) \mathrm{d}s \\
 + \int_0^T  \big( c^K(X^K_t, \alpha_t) + Z^{K,d}_t h(X^K_t, \alpha_t) \big) \mathrm{d}s - \int_0^T Z^K_s \cdot \mathrm{d}M^K_s.
\end{align*}
By definition the first integrand term is almost surely non negative and therefore we have
$$
J_0^{K, \alpha^{K*}} \geq \xi^K + \int_0^T  \big( c^K(X^K_t, \alpha_t) + Z^{K,d}_t h^K(X^K_t, \alpha_t) \big) \mathrm{d}s - \int_0^T Z^K_s \cdot \mathrm{d}M^K_s,
$$
or equivalently
$$
J_0^{K, \alpha^{K*}} \geq \xi^K + \int_0^T c^K(X^K_t, \alpha_t) \mathrm{d}s - \int_0^T Z^K_s \cdot \mathrm{d}M^{K, \alpha}_s.
$$
Taking the expectation with respect to $\mathbb{P}^{K,\alpha}$ we get the result.

\subsection{Proof of Theorem \ref{thm:pbcontinuous}}
\label{proof:pbcontinuous}

From Assumption \ref{assumption:cv_control_cont} (i)-(ii) we get that the generator $g$ satisfies the conditions of Lemma \ref{lemma:bsde_b}. Therefore $(\textbf{BSDE})$ admits a unique solution $(Y, Z)\in \mathbb{S}$. We consider $\alpha^{*}_t = \alpha^{*}(X_t,Z_t)$, and show that $\alpha^{*}$ solve the optimal control problem $(\mathbf{C})$. Since $\alpha^{*}$ is admissible according to Assumption \ref{assumption:cv_control_cont} (iii) we have $J_0^{ \alpha^{*}} = Y_0$.\\

We now take any $\alpha \in \mathcal{A}$ and show that
$$
J_0^{\alpha^{*}} \geq J_0^{\alpha}.
$$
We write:
\begin{align*}
J_0^{\alpha^{*}} = \xi + \int_0^T  \big( c(X_t, \alpha^{*}_t) + Z_t h(X_t, \alpha^{*}_t) - c(X_t, \alpha_t) - Z_t h(X_t, \alpha_t) \big) \mathrm{d}s \\
 + \int_0^T  \big( c(X_t, \alpha_t) + Z_t h(X_t, \alpha_t) \big) \mathrm{d}s - \int_0^T Z \mathrm{d}M^X_s.
\end{align*}
By definition the first integrand term is almost surely non negative and therefore we have
$$
J_0^{\alpha^{*}} \geq \xi + \int_0^T  \big( c(X_t, \alpha_t) + Z^{d}_t h(X_t, \alpha_t) \big) \mathrm{d}s - \int_0^T Z_s \cdot \mathrm{d}M^X_s,
$$
or equivalently
$$
J_0^{\alpha^{*}} \geq \xi + \int_0^T c(X_t, \alpha_t) \mathrm{d}s - \int_0^T Z_s \mathrm{d}M^{\alpha}_s.
$$
Taking the expectation with respect to $\mathbb{P}^{\alpha}$ we get the result.

\subsection{Proof of Theorem \ref{th:convergence_control}}
\label{proof:convergence_control}

According to Assumption \ref{assumption:control} the sequences $(\xi^K)_{K\geq 0}$ and $(g^K)_{K\geq 0}$ satisfy the assumptions of Lemma \ref{lemma:branching_a} for any $K$ and Assumption \ref{assumption:bsde_b}. So from Theorem \ref{th:main_theorem} we have in $\mathcal{S}^{2}_1 \times \mathcal{S}^{1}_1 \times \mathcal{S}^{1}_1$
\begin{align}
\nonumber \Big(Y^K, \int_0^{\cdot} Z^{K,d}_s &\lambda^{k,d}_s K^{-1} \mathrm{d}s , \int_0^{\cdot}  |K Z^{K,d}_s|^2 \lambda^{k,d}_s K^{-2} \mathrm{d}s \Big) \\
&\rightarrow \Big( Y, \int_0^{\cdot}  Z_s \sigma^2 X_s /2 \mathrm{d}s, \int_0^{\cdot} Z^2_s \sigma^2 X_s /2 \mathrm{d}s \Big) \text{ as } K\rightarrow + \infty. \label{eq:convergence_controls}
\end{align}

\subsubsection{Proof of point (i)}
\label{proof:convergence_control_i}

We write
\begin{align*}
\Big| \int_0^T  \alpha^{K,*}_t \lambda^{K,d}_s K^{-2} \mathrm{d}s - \int_0^T \alpha^{*}_t  \mathrm{d}A_s/2 \Big|  \leq& \int_0^T |\alpha^{K,*}_t- \alpha^{K,*}(X_t, Z_t/K)| \lambda^{K,d}_sK^{-2} \mathrm{d}s\\
&+ \Big| \int_0^T \alpha^{K,*}(X_t, Z_t/K) \lambda^{K,d}_sK^{-2} \mathrm{d}s - \int_0^T \alpha^{K,*}(X_t, Z_t/K) \mathrm{d}A_s/2 \Big| \\
& +  \int_0^T \big| \alpha^{K,*}(X_t, Z_t/K) - \alpha^*_t \big| \mathrm{d}A_t/2. 
\end{align*}
The second term converges towards $0$ by Theorem \ref{th:rescaling}. The last one terms goes to $0$ from Assumption \ref{assumption:control} (ii). Using Assumption \ref{assumption:control} (iii) we can dominate the first term by
\begin{align*}
\int_0^T |\alpha^{K,*}_t- \alpha^{K,*}(X_t, Z_t/K)| \lambda^{K,d}_sK^{-2} \mathrm{d}s &\leq A^K_T \int_0^T |\alpha^{K,*}_t- \alpha^{K,*}(X_t, Z_t/K)|^2 \lambda^{K,d}_sK^{-2} \mathrm{d}s\\
& \leq A^K_T \int_0^T \big( |X^K_t - X_t|^2 + |KZ^K_t - Z_t|^2 \big)\lambda^{K,d}_sK^{-2} \mathrm{d}s.
\end{align*}
that goes to $0$ according to Theorem \ref{th:rescaling} and to the convergence \eqref{eq:convergence_controls}.\\

In the same way, using that the control is bounded, we get that in probability
$$
\int_0^T  |\alpha^{K,*}_t|^2 \lambda^{K,d}_s K^{-2} \mathrm{d}s \rightarrow \int_0^T |\alpha^{*}_t|^2  \mathrm{d}A_s/2 \text{ as }K\rightarrow +\infty.
$$
We then extend the convergences to $\mathcal{S}^1_2$ by uniform integrability since the control is bounded. Thus we get the first statement of Theorem \ref{th:convergence_control}.\\

\subsubsection{Proof of point (ii)}
\label{proof:convergence_control_ii}

We consider $(t_i)_{1\leq i \leq n}\in [0, T]^n$ and a bounded continuous function $f$ defined from $\mathbb{R}^n$ into $\mathbb{R}$. We show that
$$
\mathbb{E}^{K, *}[f(X^K_{t_1},\dots,X^K_{t_n})] \rightarrow \mathbb{E}^{*}[f(X_{t_1},\dots,X_{t_n})] \text{ as } K\rightarrow + \infty
$$
where $\mathbb{E}^{K, *}$ (resp. $\mathbb{E}^{*}$) denotes the expectation under the control $\alpha^{K,*}$ (resp. $\alpha^*$). We write
$$
\mathbb{E}^{K, *}[f(X^K_{t_1},\dots,X^K_{t_n})] = \mathbb{E}[f(X^K_{t_1},\dots,X^K_{t_n})L^{K,\alpha^{K,*}}_T]
$$
and
$$
\mathbb{E}^{*}[f(X^K_{t_1},\dots,X^K_{t_n})] = \mathbb{E}[f(X_{t_1},\dots,X_{t_n})L^{\alpha^*}_T].
$$

Suppose we have shown that $(L^{K, \alpha^{K,*}}_T)_{K\geq 0}$ converges in probability surely towards $L^*_T$. Then writing 
$$
|L_T^{\alpha^*} - L_T^{k,\alpha^{K,*}}|  = 2(L_T^{\alpha^*} - L_T^{k,\alpha^{K,*}})_+ - (L_T^{\alpha^*} - L_T^{k,\alpha^{K,*}})
$$
we get that $(L_T^{k,\alpha^{K,*}})_{K\geq 0}$ converges towards $L_T^{\alpha^*}$ in $L^1$ by dominated converges and since 
$$
\mathbb{E}[L_T^{\alpha^*}] = \mathbb{E}[L_T^{K,\alpha^{K,*}}] = 1.
$$
Then we conclude noticing that:
$$
|f(X_{t_1},\dots,X_{t_n})L^{\alpha^*}_T - f(X^K_{t_1},\dots,X^K_{t_n})L^{K,\alpha^{K,*}}_T |\leq |f(X_{t_1},\dots,X_{t_n}) - f(X^K_{t_1},\dots,X^K_{t_n})|L^{\alpha^*}_T + \|f\|_{\infty} |L_T^{\alpha^*} - L_T^{k,\alpha^{K,*}}|
$$

We finally prove the convergence of $(L^{K, \alpha^{K,*}}_T)_{K\geq 0}$ towards $L^*_T$ in probability. We introduce the following sequences
\begin{align*}
\varepsilon^K_1 &= \int_0^T  \log\big(1+\frac{h^K(X^K_s, \alpha^{K,*}_s)}{\lambda^{K,d}_s}\big)\mathrm{d}N^{K,d}_s  - \int_0^T \frac{h^K(X^K_s, \alpha^{K,*}_s)}{\lambda^{K,d}_s} -\frac{1}{2} \Big( \frac{h^K(X^K_s, \alpha^{K,*}_s)}{\lambda^{K,d}_s}\Big)^2\mathrm{d}N^{K,d}_s,\\
\varepsilon^K_2 &= \int_0^T \Big( \frac{h^K(X^K_s, \alpha^{K,*}_s)}{\lambda^{K,d}_s}\Big)^2\mathrm{d}N^{K,d}_s - \int_0^T \big( \frac{h^K(X^K_s, \alpha^{K,*}_s)}{\lambda^{K,d}_s}\big)^2 \lambda^{K,d}_s \mathrm{d}s,\\
\varepsilon^K_3 &= \int_0^T \big( \frac{h^K(X^K_s, \alpha^{K,*}_s)}{\lambda^{K,d}_s}\big)^2 \lambda^{K,d}_s \mathrm{d}s - \int_0^T \big( \frac{h(X_s, \alpha^{*}_s)}{\sigma^2 X_s /2} \big)^2 \mathrm{d}A_s/2,\\
\varepsilon^K_4 &= \int_0^T \frac{h^K(X^K_s,\alpha^{K,*}_s )}{\lambda^{K,d}_s} \mathrm{d}M^{K,d}_s  - \int_0^T \frac{h(X_s, \alpha^{*}_s)}{\sigma^2 X_s /2}\mathrm{d}M^d_s
\end{align*}
and show that they all converges to $0$ in probability.\\

For some $C>0$ independent of $K$ we have
$$
|\varepsilon^K_1| \leq  C \int_0^T  \Big(\frac{h^K(X^K_s, \alpha^{K,*}_s)}{\lambda^{K,d}_s}\Big)^3 \mathrm{d}N^{K,d}_s \mathbf{1}_{\underset{s\in [0, T]}{\sup}\big| \frac{h^K(X^K_s, \alpha^{K,*}_s)}{\lambda^{K,d}_s} \big|<1} +   |\varepsilon^K_1| \mathbf{1}_{\underset{s\in [0, T]}{\sup} \big| \frac{h^K(X^K_s, \alpha^{K,*}_s)}{\lambda^{K,d}_s} \big|>1}.
$$
The first term of the RHS converges towards $0$ in probability according to Markov inequality. The second one since $\underset{s\in [0, T]}{\sup} \big| \frac{h^K(X^K_s, \alpha^{K,*}_s)}{\lambda^{K,d}_s} \big|$ converges almost surely towards $0$ from Assumption \ref{assumption:control}-(i). Remark that
$$
\varepsilon^K_2 = \int_0^T \Big( \frac{h^K(X^K_s, \alpha^{K,*}_s)}{\lambda^{K,d}_s}\Big)^2\mathrm{d}M^{K,d}_s
$$
Consequently using Assumption \ref{assumption:control} (iii) and  Tchebychev inequality we get that $(\varepsilon^K_2)_{K\geq 0}$ converges towards $0$ in probability. Notice that we have
\begin{align*}
\varepsilon^K_3  =&  \int_0^T \Big[ \Big( \frac{Kh^K(X^K_s, \alpha^{K,*}_s)}{\lambda^{K,d}_s}\Big)^2 -  \Big( \frac{Kh^K(X_s, \alpha^{K,*}(X_s, Z_s/K))}{\lambda^{K,d}(X_s)}\Big)^2 \Big] K^{-2}\lambda^{K,d}_s \mathrm{d}s\\
& + \int_0^T \Big( \frac{K h^K(X_s, \alpha^{K,*}(X_s, Z_s/K))}{\lambda^{K,d}(X_s)}\Big)^2 K^{-2} \lambda^{K,d}_s \mathrm{d}s - \int_0^T \Big( \frac{K h^K(X_s, \alpha^{K,*}(X_s, Z_s/K))}{\lambda^{K,d}(X_s)}\Big)^2 \mathrm{d}A_s/2 \\
&+  \int_0^T \Big( \frac{K h^K(X_s, \alpha^{K,*}(X_s, Z_s/K))}{\lambda^{K,d}(X_s)}\Big)^2 - \Big( \frac{h(X_s, \alpha^{*}_s)}{\sigma^2 X_s/2} \Big)^2  \mathrm{d}A_s/2.
\end{align*}
The second and last terms go to $0$ in probability by Theorem \ref{th:rescaling}, Assumption \ref{assumption:control} (ii) and from Proposition VI-6.12 and Theorem VI-6.22 in \cite{jacod2013limit}. As we did in the proof of Theorem \ref{th:convergence_control} (i), the first term goes to $0$ from Cauchy Schwarz inequality Assumption \ref{assumption:control} (iii) together with the convergence \eqref{eq:convergence_controls}.\\

Finally we write 
\begin{align*}
\varepsilon^K_4 \leq & \Big| \int_0^T \frac{Kh^K(X^K_s, \alpha^{K,*}_s)}{\lambda^{K,d}(X^K_s)} - \frac{Kh^K(X_s, \alpha^{*}_s)}{\lambda^{K,d}(X_s)} \mathrm{d}M^{K,d}_s \Big|\\
&+ \Big| \int_0^T \frac{K h^K(X_s, \alpha^{*}_s)}{\lambda^{K,d}(X_s)}\mathrm{d}\overline{M}^K_s - \int_0^T \frac{K h^K(X_s, \alpha^{*}_s)}{\lambda^{K,d}(X_s)} \mathrm{d}M_s \Big|\\
&+ \Big| \int_0^T \frac{h(X_s, \alpha^{*}_s)}{\sigma^2 X_s/2}-\frac{K h^K(X_s, \alpha^{*}_s)}{\lambda^{K,d}(X_s)} \mathrm{d}M_s \Big|
\end{align*}
The second and last terms converge towards $0$ by Assumption \ref{assumption:control} (ii), Theorem \ref{th:rescaling}, Proposition VI-6.12, Theorem VI-6.22 in \cite{jacod2013limit} and Theorem \ref{th:convergence_control}. Using Ito's isometry, Cauch-Schwarz inequality, Assumption \ref{assumption:control} (iii) together with Theorem \ref{th:rescaling} and Theorem \ref{th:convergence_control} (i) we get that the first term goes to $0$ in probability. Therefore $(\varepsilon^K_4)_{K\geq 0}$ converges towards $0$ in probability.\\

Thus we conclude that $(L_T^{K, \alpha^{K,*}})_{K\geq 0}$ converges toward $L^{\alpha^*}_T$ in probability.

\newpage

\appendix

\section{Spaces and notations}
\label{appendix:space}

\begin{itemize}
\item $L^p$ the set of real valued random variable $Z$ such that
$$
\|Z\|_{L_p}^p = \mathbb{E}[|Z|^p]<+\infty
$$ 
\item $\mathcal{S}_d^p$ is the set of $\mathbb{F}$-predictable $\mathbb{R}^d$ valued process $X$ such that
$$
\|X\|_{p}^p = \mathbb{E}[\underset{t\in [0, T]}{\sup} \|X_t\|^p ]<+\infty.
$$

\item For any $K\geq 0$ we consider the sets:

 \begin{itemize}
 \item $L^2(M^K)$ is the set of $\mathcal{F}^K$ predictable process $\mathbb{R}^2$ valued $Z$ such that
$$
\|Z\|_{L^2(M^K)}^2 = \mathbb{E}[\int_0^T |Z_s|^2\cdot \phi^K_s \mathrm{d}A^K_s ]<+\infty.
$$ 
\item $\mathbb{T}^K$ is the set of $\mathcal{F}^K_T$ measurable $\mathbb{R}$ valued random variable $\xi$ such that
$$
\|\xi\|_{\mathbb{T}^K}^2 = \mathbb{E}[e^{\beta A^K_T }|\xi|^2]<+\infty.
$$
\item $\mathbb{K}^K$ is the set of $\mathbb{F}^K$-optional $\mathbb{R}$ valued process $Y$ such that
$$
\|Y\|_{\mathbb{K}^K}^2 = \mathbb{E}[e^{\beta A^K_T }\underset{t\in [0, T]}{\sup}|Y_t|^2]<+\infty.
$$
\item $\mathbb{H}_2^K$ is the set of $\mathbb{F}^K$-predictable $\mathbb{R}^2$ valued process $Z$ such that
$$
\|Z\|_{\mathbb{H}^K}^2 = \mathbb{E}[\int_0^T e^{\beta A^K_t}Z_s^2 \cdot \phi^K_t \mathrm{d}A^K_t ]<+\infty \text{ with }Z^2 = (Z^2_1, Z^2_2).
$$
\item $\mathbb{H}_1^K$ is the set of $\mathbb{F}^K$-predictable $\mathbb{R}$ valued process $Y$ such that
$$
\|Y\|_{\mathbb{H}_1^K}^2 = \mathbb{E}[\int_0^T e^{\beta A^K_t}|Y_s^t|^2 \mathrm{d}A^K_t ]<+\infty.
$$
\item $\mathbb{S}^K$ is the set of pair $(Y, Z)\in \mathbb{H}_1^K\times  \mathbb{H}_2^{K}$, we note $\|(Y, Z)\|^2_{\mathbb{S}^K} = \|Y\|^2_{\mathbb{H}_1^K} + \|Z\|^2_{\mathbb{H}_2^K}$.
\end{itemize}

\item We also consider the sets related to the continuous model:
\begin{itemize}
 \item $L^2(M^X)$ is the set of $\mathcal{F}^X$ predictable process $\mathbb{R}$ valued $Z$ such that
$$
\|Z\|_{L^2(M^X)}^2 = \mathbb{E}[\int_0^T |Z_s|^2 \mathrm{d}A_s ]<+\infty.
$$ 
\item $\mathbb{T}$ is the set of $\mathcal{F}^X_T$ measurable $\mathbb{R}$ valued random variable $\xi$ such that
$$
\|\xi\|_{\mathbb{T}}^2 = \mathbb{E}[e^{\beta A_T }|\xi|^2]<+\infty.
$$
\item $\mathbb{K}$ is the set of $\mathbb{F}^X$-optional $\mathbb{R}$ valued process $Y$ such that
$$
\|Y\|_{\mathbb{K}}^2 = \mathbb{E}[e^{\beta A_T }\underset{t\in [0, T]}{\sup}|Y_t|^2]<+\infty.
$$
\item $\mathbb{H}$ is the set of $\mathbb{F}^X$-predictable $\mathbb{R}$ valued process $Z$ such that
$$
\|Z\|_{\mathbb{H}}^2 = \mathbb{E}[\int_0^T e^{\beta A_t}Z_s^2 \mathrm{d}A_t ]<+\infty.
$$
\item $\mathbb{S}$ is the set of pair $(Y, Z)\in \mathbb{K}\times  \mathbb{H}$, we note $\|(Y, Z)\|^2_{\mathbb{S}} = \|Y\|^2_{\mathbb{K}} + \|Z\|^2_{\mathbb{H}}$.
\end{itemize}

\item Finally we consider the sets:

\begin{itemize}
\item $L^p$ the set of real valued random variable $Z$ such that
$$
\|Z\|_{L_p}^p = \mathbb{E}[|Z|^p]<+\infty
$$ 
\item $\mathcal{S}_d^p$ is the set of $\mathbb{F}$-predictable $\mathbb{R}^d$ valued process $X$ such that
$$
\|X\|_{p}^p = \mathbb{E}[\underset{t\in [0, T]}{\sup} \|X_t\|^p ]<+\infty.
$$
\end{itemize}

\end{itemize}

\section{Change of measure for initial population}
\label{appendix:change_measure}

We consider $m\in \mathbb{R}_+^*$ and $n\in [0, m)$ and define the process
$$
Q^{K,n,m}_t = \int_0^t \sum_{i\in\{b,d\}} \frac{\lambda^{n,K,i}_s-\lambda^{m,K,i}_s}{\lambda^{m,K,i}_s}\mathbf{1}_{X^{m, k}_s>0}\mathrm{d}M^{K,i}_s.
$$
We have $|\Delta Q^{K,n,m}|\leq 1$ and therefore $\Delta Q^{K,n,m}\geq 1$. Moreover from Assumption \ref{assumption:model} we have for some constant $C$ positive
\begin{align*}
\langle Q^{K,n,m} \rangle_t & = \int_0^t \sum_{i\in\{b,d\}} \frac{|\lambda^{n,K,i}_s-\lambda^{m,K,i}_s|^2}{\lambda^{m,K,i}_s}\mathbf{1}_{X^{m, K}_s>0} \mathrm{d}s\\
&\leq \int_0^t \sum_{i\in\{b,d\}} \frac{C(K^2+K)|n-m|^2}{K^2 \underline{\eta}X^{m,K}_s } \mathbf{1}_{X^{m, K}_s>0} \mathrm{d}s\\
&\leq \int_0^t \sum_{i\in\{b,d\}} \frac{C(K^2+K)|n-m|^2}{K^2 \underline{\eta}X^{m,K}_{min} }  \mathrm{d}s,
\end{align*}
where $X^{m,K}_{min}>0$ is the lowest positive value that the process $X^{m,K}$ can take. Therefore by Theorem 2.4 in \cite{sokol2013optimal} the process $L^{n,m}$ is a uniformly integrable martingale.\\

Moreover according to Theorem III-3.11 in \cite{jacod2013limit} under the probability $\mathbb{P}^{K,m}$ for $i\in\{b,d\}$ the processes
$$
M^{K,m,i} - \langle Q^{K,n,m},M^{K,m,i} \rangle
$$
are local martingales. Finally we conclude since
\begin{align*}
M^{K,m,i}_t - \langle Q^{K,n,m},M^{K,m,i} \rangle_t &= N^{i}_t - \int_0^t \lambda^{m,K,i}_s - (\lambda^{n,K,i}_s - \lambda^{m,K,i}_s)\mathbf{1}_{X^{m,K}_s>0} \mathrm{d}s\\
&= N^{i}_t - \int_0^t \lambda^{n,K,i}_s \mathbf{1}_{X^{m,K}_s>0} + \lambda^{m,K,i}_s\mathbf{1}_{X^{m,K}_s=0} \mathrm{d}s\\
&= N^{i}_t - \int_0^t \lambda^{n,K,i}_s  \mathrm{d}s\\
&= M^{K,n,i}_t.
\end{align*}

\section{Admissibility of the controls in the toy model}
\label{appendix:toy_model}

\subsection{Discrete models}
\label{appendix:admissible_discrete}

We show that the control $\alpha^{K,*}$ is admissible. We have
$$
\lambda^{K,d,\alpha^{K,*}}_t = K X^{K}_{t-} (\mu + K\sigma^2) + K X^K_{t-}\big( -2a_K(t) + \frac{a_K(t)}{X^K_{t-}K}-\frac{b_K(t)}{X^K_{t-}}  \big)\mathbf{1}_{X^{K}_{t-}>0}.
$$
By \cite{jacod1975multivariate} the probability $\mathbb{P}^{K, \alpha^{K,*}}$ exists. We recall that we have chosen $T$ small enough such that $a_K$ is negative and $b_K$ positive on $[0, T]$. Hence we have
\begin{align*}
\lambda^{K,d,\alpha^{K,*}}_t & \geq KX^K_{t-} (\mu + K\sigma^2) + K \big( -2a_K(t)X^K_{t-} + \frac{a_K(t)}{K} - b_K(t)\big) \mathbf{1}_{X^K_{t-}>0} \\
&\geq X^K_{t-}K \Big( \mu -a_K(t)\mathbf{1}_{X^K_{t-}>0} + K\sigma^2 - K b_K(t) \Big).
\end{align*}
We can always assume that $T$ is small enough so that we can assume that for any $t\in [0, T]$, $\sigma^2 - b_K(t) > 0$. So $\lambda^{K,d,\alpha^{K,*}}$ is $\mathbb{P}^{K, \alpha^{K,*}}$ almost surely non negative and the control $\alpha^{K,*}$ is admissible.

\subsection{Continuous models}
\label{appendix:admissible_continuous}

We have
$$
X_t \alpha^*_t = \big( -2a(t) X_t - b(t)\big)\mathbf{1}_{X_t>0}.
$$
So the SDE
$$
\mathrm{d}X_t = X_t (\nu-\mu - \alpha^*_t) \mathrm{d}t + \sigma \sqrt{X_t} \mathrm{d}W_t
$$
writes
$$
\mathrm{d}X_t = \Big( X_t (\nu-\mu) - \big( -2a(t) X_t - b(t)\big) \mathbf{1}_{X_t>0} \Big) \mathrm{d}t + \sigma \sqrt{X_t} \mathrm{d}W_t,~X_0=x_0.
$$
Obviously this SDE admits a unique strong solution given by $Y_t\mathbf{1}_{ \underset{s\in [0, t]}{\inf}Y_s>0 }$ where $Y$ is the unique strong solution of 
$$
\mathrm{d}Y_t = \big( Y_t (\nu-\mu)  + 2a(t) Y_t + b(t) \big) \mathrm{d}t + \sigma \sqrt{Y_t} \mathrm{d}W_t,~Y_0 = x_0.
$$

\section{Feller property of the model}
\label{appendix:feller}

We consider a non negative real $x$. We obviously have that when $t\rightarrow 0$ the $X^{K, x}_t$ converges almost surely towards $x$. Now we consider a non negative sequence $(x_n)_{n\geq 0}$ that converges towards $x$ and show that for any $t>0$, $(X^{K,x_n}_t)_{n\geq 0}$ converges in law towards $X^{K,x}_t$. We fix $x_0$ larger than $x$ and any of the $x_n$ and $f$ a bounded continuous function on $\mathbb{R}_+$.\\

We write
$$
\mathbb{E}^{K,x_n}[f(X^{K,x_n}_t)] = \mathbb{E}^{K,x_0}[f(X^{K,x_n}_t)L^{K,x_n,x_0}_t] \text{ and }\mathbb{E}^{K,x}[f(X^{K,x}_t)] = \mathbb{E}^{K,x_0}[f(X^{K,x}_t)L^{K,x,x_0}_t],
$$
and
$$
|f(X^{K,x_n}_t)L^{K,x_n,x_0}_t - f(X^{K,x}_t)L^{K,x,x_0}_t| \leq |f(X^{K,x_n}_t) - f(X^{K,x}_t)|L^{K,x,x_0}_t +|f(X^{K,x_n}_t)||L^{K,x_n,x_0}_t - L^{K,x,x_0}_t|.
$$
The first term of the right hand side (RHS for short) goes to $0$ by dominated convergence. We can dominate the second on by
$$
\|f\|_{\infty} |L^{K,x_n,x_0}_t - L^{K,x,x_0}_t|
$$
that converges towards $0$ according to Scheff\'e's lemma. Therefore our model has the Feller property.

\section{Martingale representation with respect to $M^X$}
\label{appendix:martingale_representation.}

We show in this section that any $(\mathbb{F}^X, \mathbb{P})-$martingale has the representation property relative to $M^X$.\\

We set $\mathcal{H}=\mathcal{F}^X_0$ and $\mathbb{P}_0 = \varepsilon_{X_0 = x_0}$, \textit{i.e.} the probability measure on $\mathcal{H}$ such that that $\mathbb{P}_0(X_0=x_0) = 1$. For $\widetilde X$ a c\`adl\`ag process adapted to the filtration $\mathbb F^X$ and $B$ and $C$ two $\mathbb F^X$-predictable processes with finite variation such that $B_0 = C_0=0$ we recall the definition of the martingale problem associated with $(\mathcal H, \widetilde X)$ and ($\mathbb{P}_0$, $B$, $C$).

\begin{definition}[Definition III-2.6 in \cite{jacod2013limit}] A solution to the martingale problem associated with $(\mathcal H,\widetilde X)$ and $(\mathbb P_0,\widetilde B,C)$ is a probability measure $\mathbb Q$ on $(\Omega,\mathbb F^X)$ such that
\begin{itemize}
\item the restriction of $\mathbb Q$ to $\mathcal H$ equals $\mathbb P_0$,
\item $\widetilde X$ is a semi-martingale on $(\Omega,\mathbb F^X,\mathbb Q)$ with characteristics $(B,C)$.
\end{itemize}
We denote by $\mathbf{s}(\mathcal{H}, \widetilde X| \mathbb{P}_0; B, C)$ the set of solutions to this martingale problem.
\end{definition}

From this definition we see that the projection of $\mathbb{P}$ on $\mathbb{F}^K$ is a solution of $\mathbf{s}(\mathcal{H}, X| \mathbb{P}_0; D, A)$ where 
$$
D_t = \int_0^t f(X_s)\mathrm{d}s \text{ and } A_t = \int_0^t \sigma^2 X_s \mathrm{d}s.
$$
We have $M^X = X_t - D_t$ so that $M^X$ is a $\mathbb F^X$-adapted process and it makes sense to consider the set $\mathbf{s}(\mathcal{H}, M^X| \mathbb{P}_0; 0, A)$. We show that 
\begin{equation}
\label{egalite_sHXM}
\mathbf{s}(\mathcal{H}, M| \mathbb{P}_0; 0, A) = \mathbf{s}(\mathcal{H}, X| \mathbb{P}_0; D, A)
\end{equation}
and that $\mathbf{s}(\mathcal{H}, X| \mathbb{P}_0; D, A)$ is reduced to a singleton. This will be enough to conclude according to Theorem III-4.29 in \cite{jacod2013limit}.\\

Consider $\mathbb{Q}\in \mathbf{s}(\mathcal{H}, M| \mathbb{P}_0; 0, C)$. We have $X = M^X + A$ and since $D$ and $C$ are continuous process with finite variation, we deduce that under $\mathbb{Q}$ the characteristics of $X$ are $(A, C)$. Conversely, if $\mathbb{Q}\in \mathbf{s}(\mathcal{H}, X| \mathbb{P}_0; D, A)$ then by recalling that $M^X=X-D$ we obtain $\mathbb Q\in \mathbf{s}(\mathcal{H}, M| \mathbb{P}_0; 0, A)$. Hence, \eqref{egalite_sHXM} holds.\\

Since $(\mathbf{S})$ admits a unique strong solution it admits a unique solution in law (see Theorem IX-1.7 \cite{revuz2013continuous}). Therefore from Theorem III-2.26 in \cite{jacod2013limit} the set $\mathbf{s}(\mathcal{H}, X|\mathbb P_0;D,C)$ is reduced to a singleton. As a consequence of \eqref{egalite_sHXM}, the set $\mathbf{s}(\mathcal{H}, M^X| \mathbb{P}_0; 0, C)$ is also reduced to a singleton. Therefore, we deduce from Theorem III-4.29 that all $(\mathbb{F}^X, \mathbb{P})$-martingales have the representation property relative to $M^X$.

\section{Proof of Proposition \ref{prop:expo_moments}}
\label{proof:expo_moments}

For two non negative reals $\nu$ and $\mu$ we say that $N$ is a linear branching process with birth rate $\nu$ and intensity $\mu$ if it can be written as $N = N^b - N^d$ where $N^b$ and $N^d$ are two counting processes with respective intensity $\nu N$ and $\mu N$. This corresponds to a branching process as defined in Section III-3.3.1 in \cite{meleard2016modeles} with parameters $a = \nu + \mu$, $p_0 = \frac{\mu}{\nu + \mu}$ and $p_2 = \frac{\nu}{\nu + \mu}$.\\

To prove Proposition \ref{prop:expo_moments} we proceed in two steps:
\begin{itemize}
\item Step 1: We prove a result similar to Proposition \ref{prop:expo_moments} for linear branching process.
\item Step 2: We show that a under some assumption a population processes is almost surely dominated by a linear branching process.
\item Step 3: We conclude using the previous steps.
\end{itemize}

\subsection{Step 1: exponential moments for linear branching processes}
\label{proof:expo_moments_linear}

We consider $N$ a linear branching process with birth and death rate given by $\nu$ and $\mu$.\\

We define the function $F$ from $\mathbb{N} \times (\mathbb{R}_+^*)^2\times \mathbb{R}_+$ into $\overline{\mathbb{R}_+}$ by
$$
F(n, \beta, t) = \mathbb{E}_{n}[e^{\beta_1 \int_0^t N_s \mathrm{d}s + \beta_2 N_t}]
$$
where $\mathbb{E}_{n}$ is the expectation taker under the probability law that corresponds to initial condition  population of size $n$. We have the following lemma:

\begin{lemma}
\label{lemma:branching_a}
For any $\beta\in \mathbb{R}^2_+$ consider
$$
\gamma_{\nu, \mu, \beta} = \frac{\nu + \mu - \beta_1}{2\nu}, ~\phi_{\nu, \mu, \beta} = \nu \gamma_{\nu, \mu, \beta}^2 - \mu ,~\Delta_{\nu, \mu, \beta} = \sqrt{\frac{\nu}{\phi_{\nu, \mu, \beta}}}(e^{\beta_2} - \gamma_{\nu, \mu, \beta}),
$$
$$
\alpha_{\nu, \mu, \beta} = \log\big( \big| \frac{\Delta_{\nu, \mu, \beta}-1}{\Delta_{\nu, \mu, \beta}+1}\big| \big) \text{ and }t^*_{\nu, \mu, \beta} = \frac{-\alpha_{\nu, \mu, \beta}}{2\sqrt{\nu \phi_{\nu, \mu, \beta}}}.
$$
With those notations if $\beta$ satisfies $ \phi_{\nu, \mu, \beta}>0 \text{ and } \Delta_{\nu, \mu, \beta} > 1, $
we have for any $t\in [0, t^*_{\nu, \mu, \beta})$
$$
F(n, \beta, t) = \Big( \sqrt{\frac{\phi_{\nu, \mu, \beta}}{\nu}}\big( \frac{2}{1 - \text{exp}(\alpha_{\nu, \mu, \beta} + 2\sqrt{ \nu \phi_{\nu, \mu, \beta} }t)} - 1 \big) + \gamma_{\nu, \mu, \beta} \Big)^n.
$$
\end{lemma}

Note that since $(0, 0)$ satisfies the above conditions, they also hold for $\beta_1$ and $\beta_2$ small enough.

\begin{proof}[Proof of Lemma \ref{lemma:branching_a}]

Consider a population starting with one individual. We call $\tau$ the lifetime of this particle and $C=1$ or $2$ the size his offspring. Since all particles are independent and follow the same law we can consider that:
\begin{equation}
\label{eq:branching_a}
e^{\beta_1 \int_0^t N_s \mathrm{d}s + \beta_2 N_t } = \mathbf{1}_{\tau > t}e^{\beta_1 t + \beta_2} + \mathbf{1}_{\tau \leq t} e^{\beta_1 \tau } \prod_{i = 1}^{C} e^{\beta_1 \int_{\tau}^t N^{(i)}_{s-\tau} \mathrm{d}s + \beta_2N^{(i)}_{t-\tau}}
\end{equation}
where $(N^{(i)})_{1\leq i \leq 2}$ are independent copies of $N$.\\

Consider the stopping times 
$$
T_n = \inf \{s>0 \text{ s.t. } N_s=n \}\text{ and }T_n^{(i)}=\inf \{s>0 \text{ s.t. } N^{(i)}_s=n \} \text{ for } i=1,2.
$$
From Equation \eqref{eq:branching_a} we get
\begin{equation*}
e^{\beta_1 \int_0^t N^{T_n}_s \mathrm{d}s + \beta_2 N^{T_n}_t} \leq \mathbf{1}_{\tau > t}e^{\beta_1 t + \beta_2} + \mathbf{1}_{\tau \leq t} e^{\beta_1 \tau } \prod_{i = 1}^{C} e^{\beta_1 \int_{\tau}^t N^{(i)T_n^{(i)}}_{s-\tau} \mathrm{d}s + \beta_2N^{(i)T_n^{(i)}}_{t-s}}
\end{equation*}
and taking the average we have
$$
F_n(\beta, t) \leq    e^{-(\nu+\mu)t}e^{\beta_1 t + \beta_2} + \int_0^t (\nu + \mu)e^{-(\nu+\mu)s}  e^{\beta_1 s } \Big( \frac{\nu}{\nu + \mu} F^2_n(\beta, t-s)+ \frac{\mu}{\nu + \mu}   \Big)\mathrm{d}s.
$$
where $F_n(\beta, t) = \mathbb{E}_1[e^{\beta_1 \int_0^t N^{T_n}_s \mathrm{d}s + \beta_2 N^{T_n}_t}]$. We therefore consider the following ODE:
$$
(R)_{\nu, \mu, \beta}: f' = \nu f^2 - (\nu + \mu - \beta_1 )f + \mu,~f(0) = e^{\beta_2}.
$$
We show that $(R)_{\nu, \mu, \beta}$ has a unique maximal solution defined on $t\in [0, t^*_{\nu, \mu, \beta})$ by
$$
f_{\mu, \nu,\beta}(t) = \sqrt{\frac{\phi_{\nu, \mu, \beta}}{\nu}}\big( \frac{2}{1 - \text{exp}(\alpha_{\nu, \mu, \beta} + 2\sqrt{ \nu \phi_{\nu, \mu, \beta} }t)} - 1 \big) + \gamma_{\nu, \mu, \beta}.
$$

Using the change of variable $g = f - \gamma_{\nu, \mu, \beta}$, the ODE $(R)_{\nu, \mu, \beta}$ is equivalent to
$$
(R)_{\nu, \mu, \beta}':~g' = \phi_{\nu, \mu, \beta} (\frac{\nu}{\phi_{\nu, \mu, \beta}}g^2-1).
$$
By Cauchy-Lispchitz theorem this ODE admits a maximal solution $g$. By hypothesis on $\beta$ we have
$$
\phi_{\nu, \mu, \beta}\big( \frac{\nu}{\phi_{\nu, \mu, \beta}}g^2(0)-1 \big) = \phi_{\nu, \mu, \beta} (\Delta^2_{\nu, \mu, \beta}-1) >0.
$$
So for all $t$ such that $\frac{\nu}{\phi_{\nu, \mu, \beta}}g^2(t)-1>0$ we can write
\begin{equation}
\label{eq:branching_d}
\frac{g'(t)}{\frac{\nu}{\phi_{\nu, \mu, \beta}}g^2(t)-1} = \phi_{\nu, \mu, \beta}.
\end{equation}
We recognize the derivative of 
$$
x \rightarrow \frac{1}{2}\sqrt{\frac{\phi_{\nu, \mu, \beta}}{\nu}}\log\big( \frac{\sqrt{\frac{\nu}{\phi_{\nu, \mu, \beta}}}x-1}{\sqrt{\frac{\nu}{\phi_{\nu, \mu, \beta}}}x + 1} \big).
$$
So integrating on both sides of \eqref{eq:branching_d} we have
$$
\log\big( \frac{\sqrt{\frac{\nu}{\phi_{\nu, \mu, \beta}}}g(t)-1}{\sqrt{\frac{\nu}{\phi_{\nu, \mu, \beta}}}g(t) + 1} \big) = \alpha_{\nu, \mu, \beta} + 2\sqrt{\nu \phi_{\nu, \mu, \beta}}t.
$$
Therefore it is then easy to show that for any $t<t^*_{\nu, \mu, \beta}$ we have
$$
g(t)+\gamma_{\nu, \mu, \beta} = \sqrt{\frac{\phi_{\nu, \mu, \beta}}{\nu}}\big( \frac{2}{1 - \text{exp}(\alpha_{\nu, \mu, \beta} + 2\sqrt{ \nu \phi_{\nu, \mu, \beta} }t)} - 1 \big)+\gamma_{\nu, \mu, \beta}.
$$
Reciprocally it is easy to show that this function is a maximal solution of $(R)_{\nu, \mu, \beta}$ defined on $[0, t^*_{\nu, \mu, \beta})$. \\

The function $F_n(\beta, \cdot)$ being continuous a direct application of the Gr\"onwall lemma gives that for any $ t\in [0, t^*_{\mu, \nu,\beta})$, $F_n(\beta, t) \leq f_{\mu, \nu,\beta}(t)$. By monotone convergence we obtain that $F(1,\beta, t)$ is finite and taking the average in Equation \eqref{eq:branching_a} we obtain that $F(1, \beta, \cdot)$ is solution of $(R)_{\mu, \nu,\beta}$ therefore we have
$F(1, \beta,t) = f_{\mu, \nu,\beta}(t)$.\\

Finally if we consider a population $N$ starting with $n$ individual we can consider that
$$
N =  \sum_{i = 1}^n N^{(i)}
$$
where $(N^{(i)})_{1\leq i \leq n}$ are independent copies of the branching process starting with one individuals. Therefore for $t<t^*_{\nu, \mu, \beta}$ we get $F(n, \beta, t) = F(1, \beta, t)^n$ which concludes the proof.
\end{proof}

We now consider a sequence of branching process $(N^K)_{K\geq 0}$ with initial condition $K n$ and parameters 
$$
\mu^K = \mu + a K \text{ and } \nu^K = \nu + a K. 
$$
We consider $\beta_K=(\beta_1/K,\beta_2/K)$ such that $(\nu-\mu)^2>4a\beta_1$ and note
$$
\Lambda_{\infty} = \big( \frac{\nu - \mu}{2a} + \beta_2 \big) \frac{2a}{\sqrt{(\nu - \mu)^2 - 4a\beta_1}},~\eta = \frac{\sqrt{(\nu - \mu)^2 - 4a \beta_1}}{2},
$$
$$
 \alpha_{\infty} := \log( \big| \frac{\Lambda_{\infty} - 1}{\Lambda_{\infty} + 1} \big| )\text{ and }t_{\infty}^* = -\frac{\alpha_{\infty}}{2\eta}.
$$
We assume that $\beta_1$ and $\beta_2$ satisfy $ \Lambda_{\infty}>1 $. Those conditions imply that for $K$ large enough $\beta_K$ satisfies the assumption in Lemma \ref{lemma:branching_a}.\\ 

To lighten the notations we use the under-script $K$ instead of $(\nu_K, \mu_K, \beta_K)$. One can easily show the following convergence or equivalence: 
\begin{equation}
\label{eq:cv_a}
1-\gamma_K \sim \frac{\nu - \mu}{2a}\frac{1}{K},~ \frac{\phi_K}{\nu_K} \sim \frac{1}{K^2}\frac{(\nu - \mu)^2-4a\beta_1}{4a^2},
\end{equation}
\begin{equation}
\label{eq:cv_b}
\underset{K\rightarrow + \infty}{\lim } \Lambda_K = \Lambda_{\infty},~\underset{K\rightarrow + \infty}{\lim } \sqrt{\nu_K \phi_K} = \eta,~\underset{K\rightarrow + \infty}{\lim }\alpha_K = \alpha_{\infty} \text{ and } \underset{K\rightarrow + \infty}{\lim } t^*_K = t_{\infty}^*.
\end{equation}
The convergence of the sequence $(t_K)_{K\geq 0}$ implies that for any $t<t^*_{\infty}$ and $K$ large enough $F(nK, \beta_K, t)$ is finite. Moreover from \eqref{eq:cv_a} and \eqref{eq:cv_b} we get that the sequence $\big( F(nK, \beta_K, t)\big)_{K\geq 0}$ converges. More precisely it is easy to check that for any $t<t_{\infty}^*$ we have
$$
\underset{K\rightarrow + \infty}{\lim } F(nK, \beta_K, t) = e^{n\Psi(\beta, t)}
$$
where
$$
\Psi(\beta, t) = \frac{\mu - \nu}{2a} + \frac{\eta}{a} \big( \frac{2}{1-\text{exp}(\alpha_{\infty} + 2\eta t )}-1 \big).
$$

Therefore we deduce that there exists $K_0\in \mathbb{N}$, $T>0$ and $\beta\in \mathbb{R}^2_+$ such that for any $s\in [0, t)$ we have

\begin{equation}
\label{lemma:branching_c}
\underset{K\geq K_0}{\sup}\mathbb{E}_{n K}[e^{ \frac{\beta_1}{K}\int_0^s N^K_u\mathrm{d}u + \frac{\beta_2}{K}N^K_s}] < +\infty.
\end{equation}

\subsection{Step 2: domination of $X^K$ by linear process}
\label{proof:expo_moments_dmination}

We begin by showing the following lemma.
\begin{lemma}
\label{lemma:thining}
Consider two functions $g_d$ and $g_b$ from $\mathbb{R}_+$ into $\mathbb{R}_+$ such that
\[g_b(x) \leq \nu x, \; g_d(x) \geq \mu x,~g_d(0) = 0.\]
We consider two counting processes $N^{d}$ and $N^b$ with respective intensity $ g_d(N) $ and $g_b(N)$ where $N = N^b -N^d$. Then up to an extension of the probability space there exists a linear branching process $\tilde{N}$ with birth rate $\nu$ and death rate $\mu$.
\end{lemma} 
\begin{proof}[Proof of Lemma \ref{lemma:thining}]

We proceed by thinning. We consider a multivariate point process $X$ with values in $\mathcal{E} = \{b_1, b_2, d_1, d_2\}$ and let $p$ be its corresponding random measure. For any $e\in \mathcal{E}$ we define:
$$
N^{e} = \int_0^t \int_{\mathcal{E}} \mathbf{1}_{x = e} p(\mathrm{d}x,\mathrm{d}t).
$$
For $i=1$ and $2$ we note $N^{i} = N^{b_i} - N^{d_i} $ and 
$$
\lambda^{b_1} = \mu N^1,~\lambda^{d_1} = \nu N^1,~\lambda^{b_2} = g_b(N^2)\text{ and }\lambda^{d_2} = g_d(N^2).
$$
We set $ p(\mathrm{d}x, \mathrm{d}t) = m_t(\mathrm{x})\lambda_t \mathrm{d}t $ where $ \lambda_t = \lambda^{b_2}_t + \lambda^{d_1}_t $. The measure $m_{t}$ is defined by:
$$
m_t(b_1) = \varepsilon^{1}_{t} \delta_{b_1},\;m_t(b_2) = \varepsilon^1_t \varepsilon^2_t \delta_{b_2},\;
m_t(d_2) = (1 - \varepsilon^{1}_{t}) \delta_{d_2},\;m_t(d_1) = (1 - \varepsilon^{1}_{t}) \varepsilon^3_t \delta_{d_1}
$$
where $(\varepsilon^i_t)_{1\leq i \leq 3}$ are Bernoulli random variable with parameters
$$
p^1_t = \frac{\lambda^{b_1}_t}{\lambda_t},~ p^2_t = \frac{\lambda^{b_2}_t}{\lambda^{b_1}_t} \text{ and }p^3_t = \frac{\lambda^{d_1}_t}{\lambda^{d_2}_t} .
$$

For existence of the process $X$ see \cite{jacod1975multivariate}. Basically we get that the when there is an event either $N^{b_1}$ $N^{d_2}$ jump. If $N^{b_1}$ has jumped, then $N^{b_2}$ may jump or not and If $N^{d_2}$ has jumped, then $N^{d1}$ may jump or not. So almost surely we have $N^1 \geq N^2$. According to Proposition 1. in \cite{ogata1981lewis} for any $e\in \mathcal{E}$ the process $N^e$ is a counting process with intensity $\lambda^e$. This concludes the proof of the Lemma.
\end{proof}

\subsection{Step 3: conclusion}

As consequence of Lemma \ref{lemma:thining} for any $K$ up to an extension of the probability space we can consider that there exists a branching process with birth and death rate given by 
$$
\nu_K = \nu + \frac{\sigma^2}{2} K \text{ and }\mu_K = \frac{\sigma^2}{2}K
$$
that dominates $X^K$ almost surely. So according to Equation \eqref{lemma:branching_c} in Step 2, there exists some positive constants $\beta_0$, $T$ and $K_0$ such that for any $s\leq T$ 
$$
\underset{K\geq K_0}{\sup}\mathbb{E}[\text{exp}(\beta_0 \int_0^s X^K_u\mathrm{d}u + \beta_0 X^K_s)]<+\infty.
$$
This conclude the proof of the proposition.

\paragraph{Acknowledgment: }The authors are gratefully thankful to  Vincent Bansaye, Mathieu Rosenbaum, Sylvie M\'el\'eard, Dylan Possama\"i and Nizar Touzi for many interesting discussions. The authors gratefully acknowledge the financial supports of the ERC Grant 679836 Staqamof, the Chaires Analytics and Models for Regulation and Financial Risk.

\bibliographystyle{alpha}
\newcommand{\etalchar}[1]{$^{#1}$}

\end{document}